\documentclass{amsart}
\usepackage[utf8]{inputenc}

\usepackage{booktabs}

\usepackage{todonotes}
\usepackage{enumerate}
\usepackage{verbatim}
\usepackage{multirow}
\usepackage{pgf,tikz}
\usetikzlibrary{arrows}
\usepackage{tkz-graph}
\usepackage{mathtools}
\usepackage{geometry}
\usepackage{longtable} 
\usepackage{amsmath, amssymb}
\usepackage{amsthm}
\usepackage{amsfonts}
\usepackage{amscd}
\usepackage{array}
\usepackage{bigstrut}
\usepackage{breqn}
\usepackage{tikz}
\usepackage{color}
\usepackage{listings}
\usetikzlibrary{decorations.pathreplacing}
\usepackage[center]{caption}
\usepackage{enumitem}
\usepackage{amsrefs}
\usepackage{subfigure}
\usepackage{tikz,wrapfig}
\usepackage{tikz-cd}
\usepackage{graphicx}
\usepackage{caption}
\newcolumntype{L}{>{\centering\arraybackslash}p{3cm}}
\usepackage{pgf,tikz}
\usetikzlibrary{arrows}
\usepackage{tkz-graph}
\newtheorem{thm}{Theorem}[section]
\newtheorem{lem}[thm]{Lemma}

\newtheorem{cor}[thm]{Corollary}
\newtheorem{prop}[thm]{Proposition}
\theoremstyle{definition}
\newtheorem{defn}[thm]{Definition}

\theoremstyle{remark}
\newtheorem*{rem}{Remark}
\newtheorem{ex}[thm]{Example}

\newcommand{\PP}{\mathbb{P}}
\newcommand{\CC}{\mathbb{C}}

\newcommand{\ZZ}{\mathbb{Z}}
\newcommand{\QQ}{\mathbb{Q}}

\newcommand{\cE}{\mathcal{E}}
\newcommand{\cK}{\mathcal{K}}
\newcommand{\cB}{\mathcal{B}}
\newcommand{\Cstar}{\CC^{\ast}}
\newcommand{\Zfin}[1]{\ZZ/#1\ZZ}  

\newcommand{\set}[1]{\left\{#1\right\}}  

\newcommand{\inn}[1]{\left\langle#1\right\rangle}  
\newcommand{\klat}{L_{\text{K3}}}  
\newcommand{\mirror}[1]{{#1}^\vee}

\newcommand{\pp}[4]{[#1:#2:#3:#4]}  

\newcommand{\Gmax}[1]{G_{#1}} 
\newcommand{\J}[1]{j_{#1}} 
\newcommand{\SLn}[1]{\SL({#1}, \CC)}  
\newcommand{\SLgp}[1]{{\SL}_{#1}}  
\newcommand{\triv}{\{0\}}  
\newcommand{\A}[1]{A_{#1}}  

\DeclareMathOperator{\Hom}{Hom}  
\DeclareMathOperator{\Aut}{Aut}  
\DeclareMathOperator{\SL}{SL}   
\DeclareMathOperator{\rk}{rank}  
\DeclareMathOperator{\sign}{sign}
\DeclareMathOperator{\GL}{GL}
\usepackage[usenames,dvipsnames]{pstricks}
\usepackage{epsfig}
\usepackage{pst-grad} 
\usepackage{pst-plot} 
\makeatletter
\def\imod#1{\allowbreak\mkern10mu({\operator@font mod}\,\,#1)}
\makeatother

\usepackage{alltt}

\usepackage{setspace}

\title{BHK mirror symmetry for K3 surfaces with non-symplectic automorphism} 
\author{Paola Comparin, Nathan Priddis}

\begin{document}

\begin{abstract}
In this paper we consider the class of K3 surfaces defined as hypersurfaces in weighted projective space, that admit a non-symplectic automorphism of non-prime order, excluding the orders 4, 8, and 12. We show that on these surfaces the Berglund-H\"ubsch-Krawitz mirror construction and mirror symmetry for lattice polarized K3 surfaces constructed by Dolgachev agree; that is, both versions of mirror symmetry define the same mirror K3 surface.  \end{abstract}
\keywords{K3 surfaces, mirror symmetry, mirror lattices, Berglund-H\"ubsch-Krawitz construction}
\subjclass[2010]{Primary 14J28, 14J33; Secondary 14J17, 11E12, 14J32} \maketitle

\section*{Introduction}
Since its discovery by physicists nearly 30 years ago, mirror symmetry has been the focus of much interest for both physicists and mathematicians. 
Although mirror symmetry has been ``proven'' physically, we have much to learn about the phenomenon mathematically. When we speak of mirror symmetry mathematically, there are many different constructions or rules for determining when a Calabi--Yau manifold is ``mirror'' to another. The constructions are often formulated in terms of families of Calabi--Yau manifolds. 
A natural question is whether, in a situation where more than one version can apply, they produce the same mirror (or mirror family). In this article, we consider two versions of mirror symmetry for K3 surfaces, and show that in this case the answer is affirmative, as we might expect.

The first version of mirror symmetry of interest to us is known as BHK mirror symmetry. This was formulated by Berglund--H\"ubsch \cite{berghub}, Berglund--Henningson \cite{berghenn} and Krawitz \cite{krawitz} for Landau--Ginzburg models. 
Using the ideas of the Landau--Ginzburg/Calabi--Yau correspondence, BHK mirror symmetry also produces a version of mirror symmetry for certain Calabi--Yau manifolds (see Section~\ref{sec-mirror}).

In the BHK construction, one starts with a quasihomogeneous and invertible polynomial $W$
and a group $G$ of symmetries of $W$ satisfying certain conditions (see Section~\ref{quasi_sec} for more details).
From this data, we obtain the Calabi--Yau (orbifold) defined as the hypersurface $Y_{W,G}=\set{W=0}/G$.
Given an LG pair $(W,G)$, BHK mirror symmetry allows to obtain another LG pair $(W^T,G^T)$
satisfying the same conditions, and therefore another Calabi--Yau (orbifold) $Y_{W^T,G^T}$. 
We say that $Y_{W,G}$ and $Y_{W^T,G^T}$ form a BHK mirror pair. In our case, we resolve singularities to obtain K3 surfaces $X_{W,G}$ and $X_{W^T,G^T}$, which we call a BHK mirror pair. When no confusion arise, we will denote these mirror K3 surfaces simply by $X$ and $X^T$, respectively. 

Another form of mirror symmetry for K3 surfaces, which we will call LPK3 mirror symmetry, is described by Dolgachev in \cite{dolgachev}. 
LPK3 mirror symmetry says that the mirror family of a given K3 surface admitting a polarization by a lattice $M$ is the family of K3 surfaces polarized by the \emph{mirror lattice} $M^\vee$. 
We say that the two K3 surfaces are LPK3 mirror when they are lattice polarized and they belong to LPK3 mirror families (see details in Section \ref{sec-LPK3}).

Returning to the question posed earlier, one can ask 
whether the BHK mirror symmetry and LPK3 mirror symmetry produce the same mirror. 
A similar question was considered by Belcastro in \cite{belcastro}. 
She considers a family of K3 surfaces that arise as (the resolution of) hypersurfaces in weighted projective space,
uses the Picard lattice of a general member of the family as polarization, and finds that this particular polarization does not yield very many mirror families. 

This polarization fails to yield mirror symmetry for at least two reasons. First, it does not consider the group of symmetries. And secondly---and perhaps more compelling---a result proved by Lyons--Olcken (see \cite{LO}) following Kelly (see \cite{kelly}) shows that the rank of the Picard lattice of $X_{W,G}$ does not depend on $G$ at all. 
This fact suggests that we need a finer invariant than the full Picard lattice to exhibit LPK3 mirror symmetry. 
We need to find a polarizing lattice that recognizes the role of the group $G$. 

The correct polarizing lattice seems to be the invariant lattice 
\[
S_X(\sigma)=\{x\in H^2(X,\mathbb Z):\sigma^*x=x \}
\] 
of a certain non-symplectic automorphism $\sigma\in\Aut{X}$.  This was proven in \cite{ABS} and \cite{CLPS} in the case of K3 surfaces admitting a non--symplectic automorphism prime order.  

In what follows, we generalize the results of \cite{ABS} and \cite{CLPS} to K3 surfaces admitting a non--symplectic automorphism $\sigma$ of any finite order, excepting orders 4, 8 and 12. 
By polarizing each of the K3 surfaces in question by the invariant lattice $S_X(\sigma)$ of a non-symplectic automorphism $\sigma$ of finite order, we prove that BHK mirror symmetry and LPK3 mirror symmetry agree. This is done as in the previous works, by showing that $S_{X^T}(\sigma^T)$ is the mirror lattice of $S_{X}(\sigma)$.

This situation differs significantly from the case of prime order automorphism in that the invariant lattice is no longer $p$-elementary and there is no longer a (known) relationship between the invariants of $S_X(\sigma)$ and the fixed locus of $\sigma$. Hence, instead of studying the fixed locus in order to recover $S_X(\sigma)$, we determine $S_X(\sigma)$ with other methods.  
As for orders 4, 8 and 12, more details are required and the methods are slightly different, so that this will be the object of further work.

The question of whether two versions of mirror symmetry produce the same mirror has been investigated by others as well, but for different constructions of mirror symmetry than we consider here. Partial answers to the question are given by Artebani--Comparin--Guilbot in \cite{good_pairs}, where Batyrev and BHK mirror constructions are both seen as specializations of a more general construction based on the definition of good pairs of polytopes. Rohsiepe also considered Batyrev mirror symmetry in connection with LPK3 mirror symmetry in \cite{Roh}, where he shows a duality for the K3's obtained as hypersurfaces in one of the Fano toric varieties constructed by one of the 4319 3-dimensional reflexive polytopes. As in Belcastro's paper \cite{belcastro}, Rohsiepe used the Picard lattice of a general member of the family of such hypersurfaces to polarize the K3 surfaces. As it turns out, only 14 of the 95 weight systems yield a K3 surface in a Fano ambient space. We do not consider such a restriction in the current paper.

Clarke has also described a framework which he calls an auxilliary Landau--Ginzburg model, which encapsulates several versions of mirror symmetry, including Batyrev--Borisov, BHK, Givental's mirror theorem and Hori--Vafa mirror symmetry (see \cite{clarke}). Kelly also has some results in this direction in \cite{kelly}, where he shows by means of Shioda maps, that certain BHK mirrors are birational. The current article is similar in scope to these articles.

There are also several papers treating non--symplectic automorphisms of K3 surfaces, which are closely related to this paper. These include \cite{order_four} for automorphisms of order four, \cite{order_six} for order six, \cite{Schutt2010} for order $2^p$, \cite{order_eight} for order eight, and \cite{order_sixteen} for order sixteen. In general, it seems difficult to find the invariant lattice of a non--symplectic automorphism on a K3 surface. The current article gives some new methods for computing the invariant lattice, which we hope will yield more general results. 

As complementary results, in doing this classification we discovered the existence of one of the cases that couldn't be discovered in the order 16 classification in \cite{order_sixteen} namely a K3 surface admitting a purely non--symplectic automorphism of order sixteen, which has as fixed locus a curve of genus zero, and 10 isolated fixed points. This is number 58 in Table~\ref{tab-16}. Dillies has also found such an example in \cite{dillies16}. 

Additionally, our computations unearthed a different result from Dillies in \cite{order_six}. If we look at Table~\ref{tab-6}, we find the invariant lattice for number 29 and one of rows of 5d has an invariant lattice of order 12. These K3 surfaces admit an automorphism of order three, namely $\sigma_6^2$, with invariants $(g,n,k)=(0,8,5)$, but the automorphism $\sigma_6$ fixes one rational curve and 8 isolated points. This is missing from Table~1 in \cite{order_six}. Furthermore, the same can be said for the K3 surfaces in same table which have $v\oplus 4\omega_{2,1}^{1}$ as the invariant lattice, namely one of 8b, 8d, 33a, and 33b. These K3 surfaces admit a non--symplectic automorphism of order three with invariants $(g,n,k)=(0,7,4)$, but $\sigma_6$ fixes one rational curve and seven isolated points. This is also missing from the Table in \cite{order_six}. 

The paper is organized as follows. 
In Section \ref{sec-background} we recall some definitions and results on K3 surfaces and lattices, while Section \ref{sec-mirror} is dedicated to the introduction of mirror symmetry, both LPK3 and BHK. 
The main result of the paper is Theorem \ref{t:main_thm}. Section \ref{sec-method} is dedicated to the explanation of the methods used in the proof.
In Section \ref{sec-ex} we report some meaningful examples, and Section \ref{sec:tables} contains the tables proving the main theorem.

We would like to thank Michela Artebani, Alice Garbagnati, Alessandra Sarti and Matthias Sch\"utt
for many useful discussions and helpful insights. We would also thank Antonio Laface for the help on magma code \cite{magma}.
The first author has been partially supported by Proyecto Fondecyt Postdoctorado N. 3150015 and Proyecto Anillo ACT 1415 PIA Conicyt.

\section{Background}
\label{sec-background}
In this section we recall some facts about K3 surfaces and lattices. 
For notations and theorems, we follow \cite{surfaces, nikulin}

\subsection{K3 Surfaces}

A \emph{K3 surface} is a compact complex surface $X$ with trivial canonical bundle and $\dim H^1(X,\mathcal O_X)=0$. All K3 surfaces considered here will be projective and minimal. 

It is well-known that all K3 surfaces are diffeomorphic and K\"ahler. 
Given a K3 surface $X$, $H^2(X,\ZZ)$ is free of rank 22,
the Hodge numbers of $X$ are $h^{2,0}(X)=h^{0,2}(X)=1$, $h^{1,1}(X)=20$ and $h^{1,0}(X)=h^{0,1}(X)=0$, 
and the Euler characteristic is $24$. 
The Picard group of $X$ coincides with the N\'eron--Severi group, 
and both are torsion free. 

From the facts above, we see that $H^{2,0}(X)$ is one--dimensional. 
In fact, it is generated by a nowhere--vanishing two--form $\omega_X$,
which satisfies $\langle \omega_X,\omega_X \rangle=0$ and $\langle \omega_X,\overline{\omega}_X \rangle>0$. 

Given an automorphism $\sigma$ of the K3 surfaces $X$,
we get an induced Hodge isometry $\sigma^*$, which preserves $H^{2,0}(X)$, 
i.e. $\sigma^*\omega_X=\lambda_\sigma \omega_X$ for some $\lambda_\sigma\in \CC^*$. 
We call $\sigma$ \emph{symplectic} if $\lambda_\sigma=1$ 
and \emph{non-symplectic} otherwise. 
If $\sigma$ is an automorphism with nonprime order $m$, 
we say $\sigma$ is \emph{purely non-symplectic} if $\lambda_\sigma=\xi_m$ 
with $\xi_m$ a primitive $m$-th root of unity.

\subsection{Lattice theory}\label{s:lattice}

A \emph{lattice} is a free abelian group $L$ of finite rank 
together with a non-degenerate symmetric bilinear form $B\colon L \times L \to \ZZ$. 
A lattice $L$ is \emph{even} if $B(x,x)\in 2\ZZ$ for each $x\in L$.  
The \emph{signature} of $L$ is the signature $(t_+,t_-)$ of $B$. 
A lattice $L$ is \emph{hyperbolic} if its signature is $(1,\rk(L)-1)$. 
A sublattice $L\subset L'$ is called \emph{primitive} if $L'/L$ is free. 
On the other hand, a lattice $L'$ is an \emph{overlattice} of finite index of $L$ 
if $L\subset L'$ and $L'/L$ is a finite abelian group. 
We will refer to it simply as an \emph{overlattice}.

Given a finite abelian group $A$, 
a \emph{finite quadratic form} is a map $q:A\to \QQ/2\ZZ$ such that
for all $n\in \ZZ$ and $a,a'\in A$
$q(na)=n^2q(a)$ 
and 
$q(a+a')-q(a)-q(a')\equiv 2b(a,a') \imod{2\ZZ}$ 
where $b:A\times A\to \QQ/\ZZ$ is a finite symmetric bilinear form. 
We define orthogonality on subgroups of $A$ via $b$. 

Given a lattice $L$, the corresponding bilinear form $B$ induces an embedding 
$L \hookrightarrow L^\ast$, where $L^\ast:= \Hom(L,\ZZ)$. 
The \emph{discriminant group} $A_{L}:= L^\ast/L$ is a finite abelian group. 
In fact, if we write $B$ as a symmetric matrix in terms of a minimal set of generators of $L$, 
then the order of $A_{L}$ is equal to $|\det(B)|$. 
The bilinear form $B$ can be extended to $L^*\times L^*$ taking values in $\QQ$. 
If $L$ is even, this induces a finite quadratic form $q_L: A_L\to \QQ/2\ZZ$.

The minimal number of generators of $A_L$ is called the \emph{length} of $L$.
If $A_L$ is trivial, $L$ is called \emph{unimodular}. 
For a prime number $p$, $L$ is called \emph{$p$-elementary} 
if $A_L \simeq (\ZZ/p\ZZ)^a$ for some $a\in\mathbb N_0$; 
in this case, $a$ is the length of $A_L$. 

Two lattices $L$ and $K$ are said to be \emph{orthogonal}, 
if there exists an even unimodular lattice $S$ such that $L\subset S$ and $L^\perp_S\cong K$. 
Orthogonality will be a key ingredient in the definition of mirror symmetry for K3 surfaces. The following fact will also be useful. 

\begin{prop}[cf. {\cite[Corollary 1.6.2]{nikulin}}]\label{p:orth}
Two lattices $L$ and $K$ are orthogonal if and only if $q_L\cong -q_K$. 
\end{prop}

We recall the definition of several lattices that we will encounter later. 
The lattice $U$ is the hyperbolic lattice of rank 2 whose bilinear form is given by the matrix $\left( \begin{array}{cc}
0 & 1 \\
1 & 0  \\
\end{array} \right)$.

The lattices $A_n, D_m, E_6,E_7, E_8, n\geq 1, m\geq 4$ are the even negative definite lattices associated to the respective Dynkin diagrams.
For $n\geq 1$, the lattice $A_n$ has rank $n$ and its discriminant group is $\ZZ/({n+1})\ZZ$. If $p$ is prime, $A_{p-1}$ is $p$-elementary (with $a=1$). 
For $m\geq 4$, the lattice $D_m$ has rank $m$ and its discriminant group is $\ZZ/2\ZZ\oplus \ZZ/2\ZZ$ for $m$ even, and $\ZZ/4\ZZ$ for $m$ odd. 
Finally, $E_6,E_7,E_8$ have ranks 6, 7, and 8 and discriminant groups of order 3, 2, and 1, respectively. 

For $p\equiv 1 \pmod 4$ the lattice $H_p$ is the hyperbolic even lattice of rank 2, whose bilinear form is given by the matrix 
\[ 
H_p=\left( \begin{array}{cc}
\frac{(p+1)}2& 1\\
1&2\end{array} \right).
\]
The discriminant group of $H_p$ is $\ZZ/p\ZZ$. 

There are two non--isomorphic hyperbolic lattices of rank 2 with discriminant group $\ZZ/9\ZZ$ defined by the matrices
\[ L_9=\left( \begin{array}{cc}
-2& 1\\
1&4\end{array} \right),\qquad 
M_9=\left( \begin{array}{cc}
-4&5\\
5&-4\end{array} \right).\]

Following \cite{belcastro} we recall that $T_{p,q,r}$ with $p,q,r\in \ZZ$ 
is the lattice determined by a graph which has the form of a T, 
and $p,q,r$ are the respective lengths of the three legs. 
The rank of $T_{p,q,r}$ is $p+q+r-2$ and the discriminant group has order $pqr-pq-qr-pr$.

Given a lattice $L$, we denote by $L(n)$, the lattice with the same rank as $L$, but whose values under the bilinear form $B$ are multiplied by $n$. 

Many even lattices are uniquely determined by their rank and the discriminant quadratic form. 
To make this statement precise, we introduce the following finite quadratic forms. 
The notation follows \cite{belcastro} and the results are proven in \cite{nikulin}. 

We define three classes of finite quadratic forms forms, $w_{p,k}^\epsilon$, $u_k$, $v_k$ as follows:

\begin{enumerate}
\item For $p\neq 2$ prime, $k\geq 1$ an integer, and $\epsilon\in \set{\pm 1}$, let $a$ be the smallest even integer that has $\epsilon$ as quadratic residue modulo $p$. Then we define $w_{p,k}^\epsilon:\ZZ/{p^k}\ZZ\to \QQ/2\ZZ$ via $w_{p,k}^\epsilon(1)= ap^{-k}$. 

\item For $p=2$, $k\geq 1$ and $\epsilon\in \set{\pm 1, \pm 5}$, we define $w_{2,k}^\epsilon: \ZZ_{2^k}\to \QQ/2\ZZ$ on the generator via $w_{2,k}^\epsilon (1)=\epsilon\cdot 2^{-k}$. 

\item For $k\geq 1$ an integer, we define the forms $u_k$ and $v_k$ on $\ZZ/{2^k}\ZZ\times \ZZ/{2^k}\ZZ$ via the matrices:
\[
u_k=\left(\begin{matrix}
0 & 2^{-k} \\
2^{-k} & 0
\end{matrix}\right) \quad
v_k=2^{-k}\left(\begin{matrix}
2 & 1 \\
1 & 2
\end{matrix}\right)
\]

\end{enumerate}

For example, if we consider the lattice $L=A_2$, then $A_L\cong \ZZ/3\ZZ$ and $q_L$ has value $\tfrac{4}{3}$ on the generator. Thus $q_L\cong \omega_{3,1}^1$.

\begin{thm}[cf. {\cite[Thm. 1.8.1]{nikulin}}]\label{t:relations}
The forms $w_{p,k}^\epsilon$, $u_k$, $v_k$ generate the semi--group of finite quadratic forms. 
\end{thm}

In other words every finite quadratic form can be written (not uniquely) 
as a direct sum of the generators $w_{p,k}^\epsilon$, $u_k$, $v_k$. 
Relations can be found in \cite[Thm. 1.8.2]{nikulin}.

For a finite quadratic form $q$, and a prime number $p$, 
we denote $q$ restricted to the $p$--component $(A_q)_p$ of $A$ by $q_p$. 
The following results describe the close link between discriminant quadratic forms and even lattices. 

\begin{thm}[cf. {\cite[Thm. 1.13.2]{nikulin}}]\label{t:lattice_unique}
An even lattice $S$ with invariants $(t_+,t_-,q)$ is unique if, simultaneously, 
\begin{enumerate}
\item $t_+\geq 1, t_-\geq 1, t_++t_-\geq 3$;
\item for each $p\neq 2$, either $\rk S \geq 2+l((A_q)_p)$ or $
q_p\cong w_{p,k}^\epsilon\oplus w_{p.k}^{\epsilon'}\oplus q_p'$;

\item for $p= 2$, either $\rk S \geq 2+l((A_q)_2)$ or one of the following holds
\[
q_2\cong u_k\oplus q_2',\quad
q_2\cong v_k\oplus q_p',\quad
q_2\cong w_{2,k}^\epsilon\oplus w_{2.k}^{\epsilon'}\oplus q_2'.
\]

\end{enumerate}
\end{thm}

\begin{cor}[cf. {\cite[Corollary 1.13.3]{nikulin}}]\label{c:latunique}
An even lattice $S$ with invariants $(t_+,t_-,q)$ exists and is unique 
if $t_+ -t_-\equiv \sign q \imod 8$, $t_+ +t_-\geq 2+ l(A_q)$, and $t_+, t_- \geq 1$. 
\end{cor}

\begin{cor}[cf. {\cite[Corollary 1.13.4]{nikulin}}]\label{c:U+T}
Let $S$ be an even lattice of signature $(t_+,t_-)$. If $t_+\geq 1$, $t_-\geq 1$ and $t_+ +t_-\geq 3+l(A_S)$, then $S\cong U
\oplus T$ for some lattice $T$. 
\end{cor}

In Table \ref{tab-forms}, we list the discriminant form associated to each of the lattices a
ppearing in our calculations (see Sections \ref{sec-ex},\ref{sec:tables}).
A complete description can be found in \cite[Appendix A]{belcastro}.

\begin{table}[h!]\centering
\begin{tabular}{ l c c || l c c|| l c c}
 $L$&$\sign L$ &$q_L$& $L$&$\sign L$ &$q_L$& $L$&$\sign L$ &$q_L$\\
  \hline 
  $U$ 		&(1,1)	& trivial   			&$D_6$		&(0,6)	& $(w_{2,1}^1)^2$	& $T_{4,4,4}$	&(1,9)	& $v_2$\\
  $U(2)$		&(1,1)	& $u$ 			&$D_9$		&(0,9)	& $w_{2,2}^{-1}$ 	& $T_{3,4,4}$	&(1,8)	& $w_{2,3}^5$\\
  $A_1$ 		&(0,1) 	& $w^{-1}_{2,1}$ 	&$E_6$		&(0,6)	& $w_{3,1}^{-1}$	&  $T_{2,5,6}$	&(1,10)	& $w_{2,3}^{-5}$\\
  $A_2$		&(0,2)	& $w_{3,1}^{1}$ 	&$E_7$		&(0,7)	& $w_{2,1}^{1}$   	&   $<2>$		&(1,0)	& $w_{2,1}^1$\\
  $A_3$		&(0,3)	& $w_{2,2}^5$		&$E_8$		&(0,8)	& trivial			&  $<4>$		&(1,0)	& $w_{2,2}^1$\\
  $A_1(2)$		&(0,1)	& $w_{2,2}^{-1}$ 	&$H_5$		&(1,1)	& $w_{5,1}^{-1}$	&  $<8>$		&(1,0)	& $w_{2,3}^1$\\
  $D_4$		&(0,4)	& $v$ 			&$L_9$		&(1,1)	& $w_{3,2}^1$		&  $<-8>$		&(0,1)	& $w_{2,3}^{-1}$\\
  $D_5$		&(0,5)	& $w_{2,2}^{-5}$	&$M_9$		&(1,1)	& $w_{3,2}^{-1}$	&	&	&\\
 \end{tabular}
\caption{Lattices and forms} \label{tab-forms}
\end{table}

Let $L'$ be an overlattice $L'$ of the lattice $L$. 
We call $H_{L'}:=L'/L$. 
By the chain of embeddings $L\subset L'\subset (L')^*\subset L^*$ 
one has $H_{L'}\subset A_L$ and $A_{L'}=((L')^*/L)/H_{L'}$. 

\begin{prop}[cf. {\cite[Prop. 1.4.1]{nikulin}}]\label{p:overl}
The correspondence $L'\leftrightarrow H_{L'}$ is a 1:1 correspondence 
between overlattices of finite index of $L$ and $q_L$-isotopic subgroups of $A_L$,
i.e. subgroups on which the form $q_L$ is 0.
Moreover, $H_{L'}^\perp=(L')^*/L$ and $q_{L'}=({q_L}_{|H_{L'}^\perp})/H_{L'}$.
\end{prop}

\subsection{K3 lattices} \label{sec-K3lattice}
Let $X$ be a K3 surface. 
It is well--known that $H^2(X,\ZZ)$ is an even unimodular lattice of signature $(3,19)$.
As such, it is isometric to the \emph{K3-lattice} $\klat = U^3\oplus (E_8)^2$.

We let
\[
S_X = H^2(X,\ZZ) \cap H^{1,1}(X,\CC)
\]
denote the {\em Picard} lattice of $X$ in $H^2(X,\ZZ)$ and $T_X = S_X^\perp$ denote the {\em transcendental lattice}.

Let $\sigma$ be a non-symplectic automorphism of $X$. 
We let $S_X(\sigma)\subseteq H^2(X,\ZZ)$ denote the $\sigma^\ast$-invariant sublattice of $H^2(X,\ZZ)$:
\[S_X(\sigma)=\{x\in H^2(X,\ZZ): \sigma^*x=x\}.\]
One can check that it is a primitive sublattice of $H^2(X,\ZZ)$. In fact, $S_X(\sigma)$ is a primitive sublattice of $S_X$ and in general $S_X(\sigma)\subsetneq S_X$.  
We let $T_X(\sigma) = S_X(\sigma)^\perp$ denote its orthogonal complement. The signature of $S_X(\sigma)$ is $(1,t)$ for some $t\leq 19$, i.e. $S_X(\sigma)$ is hyperbolic.

\section{Mirror symmetry}
\label{sec-mirror}
\subsection{Mirror symmetry for K3 surfaces}
\label{sec-LPK3}
Mirror symmetry for a Calabi--Yau manifold $X$ and its mirror $X^\vee$ can be thought of as an exchanging of the K\"ahler structure on $X$ for the complex structure of $X^\vee$. Thus a first prediction of mirror symmetry is the rotation of the Hodge diamond:
\begin{equation*}
H^{p,q}(X,\CC)\cong H^{q,N-p}(X,\CC) 
\end{equation*}
where $N$ is the dimension of $X$.

For K3 surfaces, however, the Hodge diamond
is symmetric under the rotation mentioned above. So we need to consider a refinement of this idea. This is accomplished by the notion of lattice polarization. Roughly, we choose a primitive lattice $M\hookrightarrow S_X$, which plays the role of the K\"ahler deformations, and the mirror lattice $M^\vee$, which we now define, plays the role of the complex deformations. 
%
We will refer to this formulation of mirror symmetry simply as LPK3 mirror symmetry.  
 
Following \cite{dolgachev}, let $X$ be a K3 surface and suppose that $M$ is a lattice of signature $(1,t)$. If $j\colon M\hookrightarrow S_X$ is a primitive embedding into the Picard lattice of $X$, the pair $(X,j)$ is called an \emph{$M$-polarized K3 surface}.  There is a moduli space of $M$--polarized K3 surfaces with dimension $19-t$. 

We will not be concerned about the embedding. As in \cite{CLPS}, we will call the pair $(X,M)$ an \emph{$M$-polarizable} K3 surface if such an embedding $j$ exists.  Note that for an $M$-polarizable K3 surface $(X,M)$, the lattice $M$ naturally embeds primitively into $\klat$.

\begin{defn}\label{mirror_defn}
Let $M$ be a primitive sublattice of $\klat$ of signature $(1,t)$ with $t\leq 18$ such that $M^\perp_{L_{K3}}\cong U\oplus \mirror{M}$. 
We define $\mirror{M}$ to be (up to isometry) the \emph{mirror lattice} of $M$.\footnote{As in \cite{CLPS}, our definition in this restricted setting is slightly coarser than the one used by Dolgachev in \cite{dolgachev}, since we do not keep track of the embedding $U\hookrightarrow M^\perp$ and instead only consider $M^\vee$ up to isometry.}
\end{defn}

By Theorem \ref{t:lattice_unique} 
this definition is independent of the embedding $M$ into $\klat$. Furthermore under some conditions, (see e.g. Corollary~\ref{c:U+T} and Theorem~\ref{t:lattice_unique}) this definition is also independent of the embedding $U$ into $M^\perp$. One can check that these conditions are satisfied for the lattices we consider here. 

Note that $\mirror{M}$ also embeds primitively into $\klat$ and has signature $(1,18-t)$.  Furthermore, $q_M\cong -q_{\mirror{M}}$. One easily checks that $(\mirror{M})^\perp_{\klat}\cong U\oplus M$. 


Given $(X,M)$ an $M$-polarizable K3 surface and $(X',M')$ an $M'$-polarizable K3 surface, with $M$ and $M'$ primitive sublattices of $S_X$ and $S_{X'}$, resp., 
we say that $(X,M)$ and $(X',M')$ are \emph{LPK3 mirrors} if $M' = \mirror{M}$ (or equivalently $M=\mirror{(M')}$).

Notice that if $M$ has rank $t+1$, then the dimension of the moduli space of $M^\vee$ polarized K3 surfaces is $19-(18-t)$ which agrees with the rank of $M$. Returning to the question of K\"ahler deformations and complex deformations, we see that this definition of mirror symmetry matches the idea behind rotation of the Hodge diamond, as mentioned earlier. 


  \subsection{Quasihomogeneous polynomials and diagonal symmetries}\label{quasi_sec}
We recall a few facts and definitions (cf. \cite{CLPS} for details).
A \emph{quasihomogeneous} map of degree $d$ with integer weights $w_1, w_2, \dots, w_n$
is $W:\CC^n\to \CC$ such that for every $\lambda \in \CC$,
\[
W(\lambda^{w_1}x_1, \lambda^{w_2}x_2, \dots, \lambda^{w_n}x_n) = \lambda^dW(x_1,x_2, \dots, x_n).
\]
One can assume 
$\gcd(w_1, w_2, \dots, w_n)=1$ and say $W$ has the \emph{weight system} $(w_1, w_2, \ldots, w_n; d)$. 
Given a quasihomogeneous polynomial $W:\CC^n  \rightarrow \CC$ with a critical point at the origin,
we say it is \emph{non-degenerate} if the origin is the only critical point of $W$ and
the fractional weights $\frac{w_1}{d}, \ldots, \frac{w_n}{d}$ of $W$ are uniquely determined by $W$.
 
A non-degenerate quasihomogeneous polynomial $W$ (also called \emph{potential} in the literature) is \emph{invertible} if it has the same number of monomials as variables.

If $W$ is invertible we can rescale variables so that $W = \sum_{i=1}^n \prod_{j=1}^n x_j^{a_{ij}}$. 
This polynomial can be represented by the square matrix $\A{W} = (a_{ij})$, 
which we will call the \emph{exponent matrix} of the polynomial. 
Since $W$ is invertible, the matrix $\A{W}$ is an invertible matrix. 

The \emph{group $G_W$ of diagonal symmetries} of an invertible polynomial $W$ is 
\begin{equation*}
\Gmax{W} = \{(c_1, c_2, \ldots, c_n) \in (\CC^*)^n:  W(c_1x_1, c_2x_2, \ldots, c_nx_n) = W(x_1, x_2, \ldots, x_n)\}.
\end{equation*}

Observe that, given $\gamma=(c_1, c_2, \ldots, c_n)\in \Gmax{W}$, 
the $c_i$'s are roots of unity.
Thus one can consider $\Gmax{W}$ as a subgroup of $(\QQ / \ZZ)^n$, using addivite notation and identifying
$(c_1, c_2, \ldots, c_n) = (e^{2 \pi i g_1}, e^{2 \pi i g_2}, \ldots, e^{2 \pi i g_n})$ with $(g_1, g_2, \ldots, g_n)\in(\QQ / \ZZ)^n$.
Observe that the order of $\Gmax{W}$ is $|\Gmax{W}| = \det(\A{W})$.

Since $W$ is quasihomogeneous, the \emph{exponential grading operator} $\J{W} = \left(\frac{w_1}{d}, \frac{w_2}{d}, \ldots, \frac{w_n}{d}\right)$ is contained in $\Gmax{W}$.
We denote by $J_W$ the cyclic group of order $d$ generated by $\J{W}$: $J_W=\inn{\J{W}}$.
Moreover, each $\gamma=(g_1,\dots,g_n)$ define a diagonal matrix and thus 
$G_W$ is embedded in  $\GL_n(\CC)$. 
We define $$\SLgp{W}:=\Gmax{W}\cap \SLn{n},$$ 
i.e. $\gamma=(g_1,\dots,g_n)\in \SLgp{W}$ if and only if $\sum_i g_i\in \ZZ$.
The group $\SLgp{W}$ is called the {\em symplectic group} since, by \cite[Proposition 1]{ABS}, 
an automorphism $\sigma\in \Gmax{W}$ is symplectic if and only if $\det \sigma =1$, that is, if and only if $\sigma\in \SLgp{W}$.

\subsection{K3 surfaces from $(W,G)$}\label{K3_sec}

Reid (in an unpublished work) and Yonemura \cite{yonemura}  
have indipendently compiled a list of the 95 normalized weight systems $(w_1,w_2,w_3,w_4;d)$ 
(``the 95 families'') such that $\PP(w_1,w_2,w_3,w_4)$ 
admits a quasismooth hypersurface of degree $d$ whose minimal resolution is a K3 surface.
We consider one of these weight systems $(w_1,w_2,w_3,w_4;d)$
and an invertible quasihomogeneous polynomial of the form 
\begin{equation}\label{eq-W}
  W=x_1^m+f(x_2,x_3,x_4).  
\end{equation}
Moreover, let $G$ be a group of symmetries such that $J_W\subseteq G \subseteq \SLgp{W}$ 
and let $\widetilde{G}=G/J_W$. 
The polynomial $W$ defines a hypersurface $Y_{W,G}\subset \PP(w_1,w_2,w_3,w_4)/\widetilde{G}$
and one shows that the minimal resolution $X_{W,G}$ of $Y_{W,G}$ is a K3 surface (see \cite{ABS,CLPS}). 

The group $\Gmax{W}$ acts on $Y_{W,G}$ via automorphisms, which extend to automorphisms on the K3 surface $X_{W,G}$. 
The given form of $W$ ensures that the K3 surface $X_{W,G}$ 
admits a purely non-symplectic automorphism of order $m$:
\[
\sigma_m:\pp{x_1}{x_2}{x_3}{x_4}\mapsto \pp{\zeta_mx_1}{x_2}{x_3}{x_4}
\]
where $\zeta_m$ is a primitive $m$-th root of unity.
With additive notation, it is $\sigma_m=\left(\tfrac 1m,0,0,0\right)$.

  \subsection{BHK mirror symmetry}\label{BHK_sec}
Now we can describe the second relevant formulation of mirror symmetry coming from mirror symmetry for Landau--Ginzburg models and which we call BHK (from Berglund-H\"ubsch-Krawitz) mirror symmetry. 
This particular formulation of mirror symmetry was developed initially by Berglund--H\"ubsch in \cite{berghub}, and later refined by Berglund--Henningson in \cite{berghenn} and Krawitz in \cite{krawitz}. Because of the LG/CY correspondence and a theorem from Chiodo--Ruan \cite{BHCR}, this mirror symmetry of LG models can be translated into mirror symmetry for Calabi--Yau varieties (or orbifolds). 

We consider $(W,G)$ with $W$ invertible and $W=\sum_{i=1}^n \prod_{j=1}^n x_j^{a_{ij}}$ 
and define another pair $(W^T,G^T)$, called the BHK mirror.
We first define the polynomial $W^T$ as $$W^T = \sum_{i=1}^n \prod_{j=1}^n x_j^{a_{ji}},$$
i.e. the matrix of exponents of $W^T$ is $A_W^T$.
By the classification of invertible polynomials (cf. \cite[Theorem 1]{KrSk}), 
$W^T$ is invertible.

Next, using additive notation, one defines the dual group $G^T$ of $G$ as 
\begin{equation}\label{dualG_def}
G^T = \set{\; g \in \Gmax{W^T} \; | \;\;g\A{W} h^T \in \ZZ \text{ for all } h \in G \; }.
\end{equation}

The following useful properties of the dual group can be found in \cite[Proposition 3]{ABS}:
\begin{prop}[cf. {\cite[Proposition 3]{ABS}}]
Given $G$ and $G^T$ as before, one has:
\begin{enumerate}
\item $(G^T)^T = G$.
\item If $G_1\subset G_2$, then $G_2^T\subset G_1^T$ and $G_2/G_1\cong G_1^T/G_2^T$.
\item $(\Gmax{W})^T = \triv$, $(\triv)^T = \Gmax{W^T}$.  
\item $(J_W)^T=\SLgp{W^T}$. In particular, if $J_W\subset G$, then $G^T\subset \SLgp{W}$. 
\end{enumerate}
\end{prop}
Given the pair $(W,G)$ with $W$ invertible with respect to one of the 95 weight systems, 
we associated to it the K3 surface $X_{W,G}$.
One can  check that in this case the weight system of $W^T$ also belongs to the 95.
By the previous result, $J_{W^T}\subseteq G^T\subseteq \SL_{W^T}$, so that $X_{W^T,G^T}$ is again a K3 surface. 
We call $X_{W^T,G^T}$ the \emph{BHK mirror of $X_{W,G}$}.

\subsection{Main theorem}\label{main_sec}
We have  described two kinds of mirror symmetry for K3 surfaces: LPK3 mirror symetry and BHK one. 
Since mirror symmetry describes a single physical phenomenon, 
we expect the two constructions to be compatible in situations where both apply. 
We will now state our main theorem, which shows that BHK and LPK3 mirror symmetry agree for the K3 surfaces $X_{W,G}$, when $W$ is of the form \eqref{eq-W}. 
When no confusion arises, we will denote the mirror K3 surfaces $X_{W,G}$ and $X_{W^T,G^T}$ simply by $X$ and $X^T$. 

Consider the data $(W,G,\sigma_m)$, 
where \begin{itemize}
\item $W$ is an invertible polynomial of the form \eqref{eq-W} 
whose weight system belong to the 95 families of Reid and Yonemura, 
\item $\sigma_m=(\frac 1m,0,0,0)$ is the non-symplectic automorphism of order $m$, 
\item $G$ is a group of diagonal symmetries of $W$ such that $J_W\subseteq G\subseteq \SL_W$.\end{itemize}
By section \ref{sec-K3lattice}, the invariant lattice $S_X(\sigma_m)$ is a primitive sublattice of $S_X$
and $(X_{W,G}, S_X(\sigma_m))$ is a $S_X(\sigma_m)$--polarizable K3 surface.
Let $r$ be the rank of $S_X(\sigma_m)$. 
The BHK mirror is given by $(W^T,G^T,\sigma_m^T)$, where $\sigma_m^T$ is the non-symplectic automorphism of order $m$ on $X_{W^T,G^T}$. 
Notice that $\sigma_m$ and $\sigma_m^T$ have the same form, namely $(\tfrac{1}{m},0,0,0)$, but they act on different surfaces.

\begin{thm}\label{t:main_thm}
Suppose $m\neq 4,8,12$. If $W$ is a polynomial of the form \eqref{eq-W}, quasihomogeneous with respect to one of the 95 weight systems for K3 surfaces as in Section~\ref{K3_sec} and $G$ is a group of diagonal symmetries satisfying $J_W\subseteq G\subset\SL_W$, then $\left(X_{W^T,G^T}, S_{X^T}(\sigma_p^T\right))$ is an LPK3 mirror of $\left(X_{W,G}, S_X(\sigma_p)\right)$. 
\end{thm}

The theorem is proved by showing that $$S_X(\sigma_m)^\vee\cong S_{X^T}(\sigma_m^T).$$
 As we have seen in section \ref{s:lattice}, this amounts to checking that the invariants $(r,q_{S_X(\sigma_m)})$ for $X_{W,G}$ and $(r^T,q_{S_{X^T}(\sigma_m^T)})$ for $X_{W^T,G^T}$ satisfy $r=20-r^T$ and $q_{S_X(\sigma_m)}\cong -q_{S_{X^T}(\sigma_m^T)}$. Thus the heart of the proof is determining $q_{S_X(\sigma_m)}$ (or equivalently in our case $S_X(\sigma_m)$). 
In the following section, we will describe how this is done. 
It involves computing the invariant lattice and its overlattices. 
We list the results in tables in Section \ref{sec:tables}. 
Unfortunately, our method does not work for $m=4,8,12$ due to the presence of many overlattices, so that we cannot exactly pinpoint the invariant lattice.

\section{Methods}
\label{sec-method}
In the setting of Theorem \ref{t:main_thm}, one has to show that $S_X(\sigma_m)^\vee\cong S_{X^T}(\sigma_m^T)$.
Whenever $m=p$ a prime number, Theorem \ref{t:main_thm} was proved using a similar method in \cite{ABS} for $m=2$ and \cite{CLPS} for other primes. 
There is not a general method of proof in either article; instead the theorem is checked in every case.  

In \cite{ABS} and \cite{CLPS} there are several tools introduced in order to facilitate computation of the invariant lattice.  
The proof we give here follows roughly the same idea, however the methods used in the previous articles for computing $S_X(\sigma_m)$ are no longer valid, when $m$ is not prime. 
In order to illustrate the differences, we highlight briefly the method used in case of $p$ prime. 
Then we will describe the proof of the theorem, in case $m$ is not prime.

\subsection{Method for $m=p$ prime}\label{s:methodsprime}

As mentioned, the argument given in \cite{ABS} and \cite{CLPS} essentially boils down to determining the invariant lattice $S_X(\sigma_p)$ for $X_{W,G}$. 
The method for determining this lattice relies on the following powerful theorems. 

\begin{thm}[\cite{other_primes}] Given a $K3$ surface with a non-symplectic automorphism $\sigma$ of order $p$, a prime, the invariant lattice $S_X(\sigma)$ is $p$--elementary, i.e. $A_{S_X(\sigma)}\cong (\ZZ/p\ZZ)^a$. 
\end{thm}

\begin{thm}[\cite{RS,nikulin}] For a prime $p\neq 2$, a hyperbolic, $p$--elementary lattice $L$ with rank $r\geq 2$ is completely determined by the invariants $(r,a)$, where $a$ is the length. An indefinite 2--elementary lattice is determined by the invariants $(r,a,\delta)$, where $\delta\in{0,1}$ and $\delta=0$ if the discriminant quadratic form takes values 0 or 1 only and $\delta=1$ otherwise.
\end{thm}
By Proposition \ref{p:orth}, the orthogonal complement in $L_{K3}$ of a $p$--elementary lattice with invariants $(r,a)$ is a $p$--elementary lattice with invariants $(22-r,a)$. For both a given 2--elementary lattice and its orthogonal complement, the third invariant $\delta$ agrees. 
In the setting of Theorem \ref{t:main_thm}, Corollary \ref{c:U+T} shows that we also have $S_X(\sigma_p)^\perp\cong U\oplus M$ where $M$ is a hyperbolic $p$--elementary with invariants $(20-r,a)$. Thus for $p\neq 2$ it is enough to verify that $(r,a)$ for $X_{W,G}$ and $(r^T,a_T)$ for $X_{W^T,G^T}$ satisfy $r=20-r^T$ and $a=a_T$. For $p=2$ the third invariant $\delta$ must also be compared, which the authors checked in \cite{ABS}. 

In order to compare $(r,a)$, we first look at the topology of the fixed point locus. 

\begin{thm}[cf. \cite{nikulin2, other_primes}]\label{summary_invariants}
Let $X$ be a K3 surface with a non-symplectic automorphism $\sigma$ of prime order $p\neq 2$. Then the fixed locus $X^\sigma$ is nonempty, and consists of either isolated points or a disjoint union of smooth curves and isolated points of the following form:
\begin{equation}
X^\sigma=C\cup R_1\cup\ldots\cup R_k\cup \{p_1,\ldots,p_n\}\label{e:fixedlocus}.
\end{equation}
Here $C$ is a curve of genus $g\geq 0$, $R_i$ are rational curves and $p_i$ are isolated points. 

If $p=2$, then the fixed locus is either empty, the disjoint union of two elliptic curves, or is of the form \eqref{e:fixedlocus} with $n=0$. 
\end{thm}

In \cite{ABS} and \cite{CLPS}, $\sigma_p$ always fixes a curve. Furthermore, the case of two elliptic curves does not appear in this setting described there. Therefore, the fixed locus is determined by the invariants $(g,k,n)$.  
In \cite{other_primes}, the authors give formulas to calculate $(r,a)$ given $(g,k,n)$ for each prime $p$ (see \cite[Theorem 0.1]{other_primes}).
%
%
%
%
Thus, in order to prove Theorem \ref{t:main_thm} for $p$ prime, one first computes the invariants $(g,k,n)$,
and from them computes the invariants $(r,a)$ (if $p=2$, additional computation are required to obtain $\delta$). 
Then one compares the invariants for BHK mirrors as described above.

\subsection{Method for $m$ not prime}\label{sec-notprime}

If $m$ is not prime, $S_X(\sigma_m)$ is no longer $p$--elementary for any prime $p$. 
The best we can say about $A_{S_X(\sigma)}$ is that it is a subgroup of $(\ZZ_m)^a$, as in the following lemma.  

\begin{lem}
Let $\sigma$ be a purely non-symplectic automorphism of order $m$. The discriminant group $A_{S_X(\sigma)}$ is a subgroup of $(\ZZ_m)^a$ for some $a\geq 0$. 
\end{lem}

\begin{proof}
We know $A_{S_X(\sigma)}$ is a finite abelian group. Thus it suffices to show that for each $x\in A_{S_X(\sigma)}$, we have $mx=0$. By Maschke's Theorem, the vector space $H^2(X,\ZZ)\otimes \QQ$ is completely reducible. Therefore it decomposes as a sum of irreducible $\QQ[\ZZ/m\ZZ]$-modules. Each direct summand is of the form $Q(\xi_d)$ where $d|m$. So we have $\sum_{i=0}^{m-1} (\sigma^*)^i=0$ on $T_X(\sigma)$. Now we have $A_{S_X(\sigma)}\cong H^2(X,\ZZ)/(S_X(\sigma)\oplus T_X(\sigma))\cong A_{T_X(\sigma)}$, and these automorphisms commute with the action of $\sigma^*$. Thus $\sigma^*$ acts by identity on $A_{T_X(\sigma)}$, but since $\sum_{i=0}^{m-1} (\sigma^*)^i=0$ on $T_X(\sigma)$, $mx=0 \imod {T_X(\sigma)}$ for $x\in T_X(\sigma)^*$. 
\end{proof}

We see that $A_{S_X(\sigma_m)}$ can have any possible order dividing a power of $m$. Even in the best of situations (e.g. $A_{S_X(\sigma_m)}\cong \ZZ_m^a$) the lattice is no longer determined by $(r,a)$. 
Consider for example the two lattices $L_1=U\oplus D_5^2$ and $L_2=U\oplus A_3\oplus D_7$: they both have discriminant group $\ZZ_4^2$ and $(r,a)=(12,2)$, but $q_{L_1}=(w_{2,2}^3)^2$ whereas $q_{L_2}=w_{2,2}^1\oplus w_{2,2}^5$. In other words, we have two different lattices with the same rank, and isomorphic discriminant groups. 

It is true, however, that the lattice is determined in all cases considered here by 
the rank and the discriminant quadratic form 
(cf. Corollary \ref{c:latunique}). 
Thus, rather than computing $(r,a)$ it suffices to compute $(r,q_{S_X(\sigma_m)})$. 
We will be able to compute these invariants for all K3 surfaces $X_{W,G}$ with $W$ as in \eqref{eq-W} and $G\subset \SL_W$ except in case $m=4,8,12$. 

The first step to compute  $S_X(\sigma_m)$ is to compute its rank $r$.
In order to do so, we describe an important set of curves on $X_{W,G}$, the orbits of the exceptional curves.   
Consider the hypersurface $Y_{W,G}$ as in Section \ref{BHK_sec}, and the resolution of singularities $X_{W,G}\to Y_{W,G}$. Because $W$ is nondegenerate, all of the singularities of $Y_{W,G}$ lie on the coordinate curves. 

Let $\cE$ denote the set of the exceptional curves. Because $\sigma_m$ leaves the coordinate curves invariant, the action of $\sigma_m$ on $X_{W,G}$ induces an action of $\sigma_m$ on $\cE$. Given an exceptional curve $E$, let $G_E$ denote the isotropy group for $E$ and let $\cE/\sigma_m$ denote the set of orbits of this action. Now we set 
\[
b_E=\tfrac 1{|G_E|}\sum_{i=0}^{m-1}\sigma_m^iE.
\]
This is simply the sum of all curves in the orbit of $E$. In particular, if $E$ is fixed by $\sigma_m$, then $b_E=E$. We can use this topology to determine $r$: 

\begin{lem}\label{l:rank}
The rank $r$ of $S_X(\sigma)$ is equal to 1 plus the number of orbits of exceptional curves in the blow-up $X_{W,G}\to Y_{W,G}$, i.e. 
\[
\rk S_X(\sigma_m)=1+|\cE/\sigma_m|
\]
\end{lem}

\begin{proof}

Consider the quotient $X_{W,G}/\sigma_m$. 
We have the following diagram


\[  
\begin{tikzcd} 
X_{W,G}   \arrow{r}{} \arrow{d}{\pi}  &  Y_{W,G} \arrow{d}{\pi}\\
X_{W,G}/\sigma_m\arrow{r}{}  &  \PP(w_2,w_3,w_4) \\
\end{tikzcd}
\]
where the bottom horizontal arrow is obtained by blowing up isolated singular points of $\PP(w_2,w_3,w_4)$ and singular points on the branch locus $\set{f(x_2,x_3,x_4)=0}\subset \PP(w_2,w_3,w_4)$. Notice in the quotient $\PP(w_2,w_3,w_4)\cong Y_{W,G}/\sigma_m$, the branch locus corresponds to the curve $y=0$ in $Y_{W,G}$. Furthermore, each orbit of exceptional curves will correspond to exactly one exceptional curve of the bottom horizontal map. 

Since $h^2(\PP(w_2,w_3,w_4))=1$ (see e.g. \cite{twistedproj} or \cite{dolgachev2}), and $H^2(X_{W,G}/\sigma_m,\QQ)$ is generated by the exceptional curves and the pullback from $\PP(w_2,w_3,w_4)$, we have 
\[
\dim H^2(X_{W,G}/\sigma_m,\QQ)=1+|\cE/\sigma_m|.
\]

Finally, we know by \cite{macdonald} that $H^2(X_{W,G}/\sigma_m,\QQ)=H^2(X_{W,G}, \QQ)^{\sigma_m}$. This implies the result.  
\end{proof}

\begin{rem}
Although $H^2(X_{W,G}, \QQ)^{\sigma_m}$ is generated by the orbits of the exceptional curves and one coordinate curve, this is not always the case for the lattice $S_X(\sigma_m)$.
\end{rem}

The second step in determining $(r,q)$ requires another set of curves. 
Let $\cK$ be the set of smooth irreducible components of the (strict transforms of the) coordinate curves ${x_i=0}$
(we omit curves that have worse singularities than a separating nodes between two or more smooth components).
As before $\sigma_m$ acts on this set. 
For $C\in \cK$, we define $G_C$ as the isotropy group, $\cK/\sigma_m$ as the set of orbits, and 
\[
b_C=\tfrac 1{|G_C|}\sum_{i=0}^{m-1}\sigma_m^iC.
\]
Notice that $\sigma_m$ leaves the coordinate curves invariant, so $b_C$ is the sum of some of the smooth irreducible components of one of the coordinate curves. In cases where the coordinate curve is smooth, $b_C$ is the coordinate curve. There are cases, however, where the coordinate curve is not irreducible.

Finally we define the set of orbit curves $\cB=\set{b_E\mid E\in \cE/\sigma_m}\cup \set{b_C\mid C\in \cK/\sigma_m}$. Let $L_{\cB}$ denote the sublattice of $S_X$ generated by $\cB$. This lattice is clearly a sublattice of $S(\sigma_m)$ since $\sigma_m$ fixes each of the generators. Therefore, the rank is at most $r$. 
By Lemma~\ref{l:rank}, we see that $\rk L_\cB=r$. Hence $S(\sigma_m)$ is an overlattice of $L_\cB$. Because of this explicit description, we can actually determine the lattice $L_{\cB}$ in each case. (See the appendix for Magma code, which allows us to compute $L_\cB$.)  

Determining $(r,q)$ now breaks into four different cases, each of which requires a different method. 
We now describe those four methods. In most cases, we show $L_\cB=S_X(\sigma_m)$.

\subsubsection{Method I}
The first method is the most common. 
Using Proposition~\ref{p:overl}, we can determine all of the overlattices of $L_\cB$. It often happens that the discriminant group has no isotropic subgroups, and so there are no nontrivial overlattices. Since $S_X(\sigma_m)$ is an overlattice of $L_\cB$, we can conclude $S_X(\sigma_m)=L_\cB$. An example of this situation is shown in Example \ref{ex-method1}. 

This covers all possibilities when $m=18,20,24,30,42$, and parts of the other orders, namely whenever $m$ and $r$ are one of the following possibilities: $m=10$,  $r=3,8,12,17$; $m=9$, $r=4,16$; and $m=6$, $r=1,4,6,9,11,14,16,19$. One may notice in Table~\ref{tab-6} that there are two possibilities for  $m=6, r=10$. One of them has quadratic form $4\omega_{3,1}^1$ and therefore has no overlattices, so this method works. But the other $v+4\omega_{2,1}^{-1}$ does have overlattices, so we use another method.

\subsubsection{Method II}
The second method does not use $L_\cB$. Instead, we notice that some power $\sigma_m^k$ has order $p$ for some prime $p$ and $S_X(\sigma_m)\subset S_X(\sigma_m^k)$ is a primitive embedding of lattices. If it happens that $\rk S_X(\sigma_m)= \rk S_X(\sigma_m^k)$, then indeed $S_X(\sigma_m)= S_X(\sigma_m^k)$. In this case $S_X(\sigma_m)$ is $p$--elementary and we can use the methods in Section~\ref{s:methodsprime} to compute it.  
See Example~\ref{ex-method2} for an example of this situation.

This covers the remaining cases with $m=6,10, 15, 22$, the cases with $m=9$, $r=16$, as well as the two cases with $m=14$, $r=13$.  

\subsubsection{Method III} 
For the third method, we use the results of Belcastro in \cite{belcastro}. Belcastro has computed the Picard lattice $S_{X_t}$ of the general member of a family of K3 surfaces $X_t$ for each of the 95 weight systems. Each invertible polynomial $W$ yields a special member of this family, but often with bigger Picard lattice. It often happens that $S_{X_t}\cong L_\cB$. When this happens, because $S_{X_t}\subset H^2(X_t,\ZZ)$ is a primitive embedding of lattices, the composition of embeddings 
\[
L_\cB\subset S_X\subset H^2(X_W,\ZZ)
\]
remains a primitive embedding for the special member, i.e. $L_\cB$ is primitively embedded in $H^2(X_W,\ZZ)$. Here we require the group $G=J_W$. 

By the isomorphism theorems for groups, $S_{X_W}/L_\cB\subset H^2(X_W.\ZZ)/L_\cB$, and the latter is a free abelian group. Therefore $L_\cB$ is a primitive sublattice of $S_X$. By similar argument, $L_\cB$ is a primitive sublattice of $S_X(\sigma_m)$. Since they have the same rank, we see $L_\cB=S_X(\sigma_m)$. We illustrate this method in Example~\ref{ex-meth3}.

This method is used whenever $m$ and $r$ are one of the following possibilities: $m=9$, $r=8$; $m=14$, $r=7$; and $m=16$, $r=9$. 

\subsubsection{Method IV} This leaves two cases:  $m=16$, $r=11$ and $m=9$, $r=12$. As each of these cases is different we will work out each case in Examples ~\ref{ex-meth41} and \ref{ex-meth42}. This method involves finding $S_{X_{W,G}}$ and showing explicitly that $L_\cB$ embeds primitively into $S_{X_{W,G}}$. Then, as in Method II, we can conclude that $L_\cB=S_X(\sigma_m)$.

\begin{rem}
What we see when employing these four methods is that the invariant lattice seems to be the lattice $L_\cB$. This is certainly true for every invariant lattice computed by methods I,III and IV, though it also seems to be true for each of the others (including the invariant lattice for the automorphisms of prime order and of order 4, 8, 12), though we have not checked every case explicitly. 

\end{rem}

In what follows, we will make use of the following lemma. 

\begin{lem}[cf. \cite{order_six, order_four}] \label{l:rationaltree}
Let $R_1,R_2,\dots, R_s$ be a tree of smooth rational curves on a K3 surface $X$ and $\sigma$ a non-symplectic automorphism of finite order on $X$ leaving each of the $R_i$ invariant. Then the intersection points of the $R_i$'s are fixed by the automorphism and it suffices to know the action at one intersection point to know the action on the entire tree.
\end{lem}

We will also make use of the following formula, giving the genus of a curve of degree $d$ in weighted projective space $\PP(w_1,w_2,w_3)$ in order to compute the genus of the coordinate curves. This formula can be found, e.g., in \cite[Theorem 12.2]{fletcher}.
\begin{equation*}
g(C)=\frac{1}{2}\left(\frac{d^2}{w_1w_2w_3}-d\sum_{1\leq j<i}\frac{\gcd(w_i,w_j)}{w_iw_j}+\sum_{i=1}^3\frac{\gcd(w_i,d)}{w_i}-1\right).
\end{equation*}
With the genus, one can also compute self--intersection numbers for curves on a K3 surface via Riemann--Roch:
\[
2-2g(C)= C\cdot C
\]

\subsection{Proof of main Theorem} 
We now provide the details of the proof of Theorem \ref{t:main_thm}.
For each K3 surface $X_{W,G}$ we have the invariants $(r,q_{S_X(\sigma_m)})$ for the invariant lattice $S_X(\sigma_m)$, as discussed in the previous section. We know the invariant lattice has signature $(1,r-1)$, and so the orthogonal complement has signature $(2,20-r)$. We check that the conditions of Corollary \ref{c:U+T} are fulfilled so that $S_X(\sigma_m)^\perp_{L_{K3}}\cong U\oplus M$ for some lattice $M$\footnote{In one case, the conditions are not fulfilled, namely $m=6$, $r=1$. However, in this case, we know the lattice $U\oplus\langle4\rangle$ has the given invariants, and by Corollary~\ref{c:latunique} this is the only such lattice}. This lattice $M$ is hyperbolic with signature $(1,19-r)$ and has discriminant quadratic form $-q_{S_X(\sigma_m)}$. One can see from the tables that the conditions of Theorem~\ref{t:lattice_unique} are satisfied. Hence, there is exactly one lattice with these invariants. To prove the theorem, therefore, we need simply to check that $(20-r, -q_{S_X(\sigma_m)})$ are the invariants for the invariant lattice $S_{X^T}(\sigma_m^T)$ for $X_{W^T,G^T}$. This can be checked by consulting the tables in Section \ref{sec:tables}. 
This concludes the proof. 

Tables contain all possible invertible polynomial of the from \eqref{eq-W} with non--symplectic automorphism of order $m$, and for each polynomial, we list the orders of the possible groups $G/J_W$ satisfying $J_W\subset G\subset \SL_W$.
In most cases $\SL_W/J_W$ is cyclic, so the properties of $G^T$ make it clear what the dual group is. However for two examples, one can see by the multiplicity of subgroups with the same order that the group $\SL_W/J_W$ is not cyclic. These two examples are $x^3+y^3+z^6+w^6$ (number 3d) and $x^2+y^4+z^6+w^{12}$ (number 8d) in Table \ref{tab-6}. In these cases, we will clear up any ambiguities in the following sections.

From now on we will make a change of notation from $(x_1,x_2,x_3,x_4)$ to $(x,y,z,w)$ for the variables of $W$, so that the variables are arranged with the weights in nonincreasing order. In other words, it is possible that $x_1$ corresponds to any of $x,y,z$, or $w$. This convention is also used in Tables of Section \ref{sec:tables}.

\begin{rem}
It is possible that a given K3 surface admits a purely non--symplectic automorphism of different orders. It turns out that it doesn't matter which automorphism one uses to exhibit LPK3 mirror symmetry, the notion still agrees with BHK mirror symmetry, in the sense of Theorem~\ref{t:main_thm}, as long as the defining polynomial is of the proper form \eqref{eq-W}. We expect that the theorem still holds for K3 surfaces that don't take the form of \eqref{eq-W}, but that is a topic for further investigation. 
\end{rem}

\section{Examples}
\label{sec-ex}
In this section we will first give examples to illustrate each of the methods that was used to determine $S_X(\sigma_m)$. Then we will describe the subgroups in the two cases where $\SL_W/J_W$ is not cyclic.

\begin{ex}\label{ex-method1}Method I:

This first example will illustrate Method I for determining $S_X(\sigma_m)$. Let us consider the K3 surface with equation 
\[
W=x^2+y^3+z^9+yw^{12}=0
\] 
in the weighted projective space $\PP(9,6,2,1)$ with degree 18. This is number 12b in Table~\ref{tab-9}. There are two non--symplectic automorphisms of interest $\sigma_2:(x,y,z,w)\mapsto(-x,y,z,w)$ and $\sigma_9:(x,y,z,w)\mapsto(x,y,\mu_9z,w)$. The invariant lattice $S_X(\sigma_2)$ was dealt with in \cite{ABS}, so we focus on $\sigma_9$.


Here $|G_W|=36\cdot 18$, $|J_W|=18$ and the weight system for the BHK mirror is $(18,11,4,3;36)$ so that $[G_{W}:\SL_W]=36$. Thus $|\SL_W/J_W|=1$.

Looking at the action of $\Cstar$ on the weighted projective space $\PP(9,6,2,1)$, we find the following isotropy:
\begin{align*}
\mu_3 &: \text{fixes }z=w=0, x^2+y^3=0\\
\mu_2 &: \text{fixes }x=w=0, y^3+z^9=0.
\end{align*}
The first row provides a single point with $\Zfin{3}$ isotropy ($A_2$ singularity), and the second provides three points each with $\Zfin{2}$ isotropy (3 $A_1$'s). 
Their resolution gives the configuration of curves on $X_{W,G}$ depicted in Figure \ref{fig:res}. In this depiction, we have not indicated the three intersection points between $C_x$ and $C_z$.

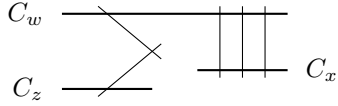
\begin{figure}[ht]
\centering\begin{tikzpicture}[xscale=.6,yscale=.5]
\draw [thick](-1,0)--(1,0);
\node [left] at (-1.2,0){$C_z$};
\draw [thick](-1,2)--(4,2);
\node [left] at (-1.2,2){$C_w$};
\draw [thick] (2,0.5)--(4,0.5);
\node [right] at (4.2,0.5){$C_x$};
\draw (-0.2,-0.2)--(1.2,1.2);
\draw (1.2,0.8)--(-0.2,2.2);
\draw (3.5,0.3)--(3.5,2.2);
\draw (3.0,0.3)--(3.0,2.2);
\draw (2.5,0.3)--(2.5,2.2);
\end{tikzpicture}
\caption{Resolution of singularities on $X_W$}
\label{fig:res}
\end{figure}


The set $\cE$ consists of these five exceptional curves. Denote by $E_1$, $E_2$ and $E_3$ the three $A_1$ fibers, and $E_4$ and $E_5$ the two curves in the $A_2$ fiber. Looking at the form of $W$, we see that the curves $C_x=\{x=0\}$, $C_z=\{z=0\}$ and $C_w=\{w=0\}$ are smooth. The curve $C_y=\{y=0\}$ is not smooth. Thus the set $\cK$ consists of these three smooth curves. The curve $C_x$ has genus 7, $C_z$ has genus 1, and $C_w$ 
has genus 0.  

There are two important representatives of the coset $\sigma_9J_W$ in $G_W$ which will help us compute the fixed locus for $\sigma_9$, namely $(0,0,\tfrac 19,0)$ and $(0,\tfrac 23,0,\tfrac 49 )$. These representatives show us that the curve $C_z=\{z=0\}$ is fixed and the point defined by $\{w=y=0, x^2+z^9=0\}$ is also fixed. 
Because $C_z$ is fixed, $E_4$ and $E_5$ are invariant (though not fixed pointwisely). 
The point of intersection of $C_w$ and the $A_2$ exceptional fiber and this other point are the only fixed points on $C_w$. Thus the three $A_1$ singularities are permuted by the action. By Lemma \ref{l:rank}, since there are three orbits, $S_X(\sigma_9)$ has rank 4. 

We now compute the lattice $L_\cB$, generated by $\cB=\set{E_1+E_2+E_3, E_4,E_5, C_x,C_z,C_w}$. Since there are six generators, two of them are redundant, for example $C_x$ and $C_z$. 
Consider the lattice $L$ generated by $E_1+E_2+E_3$, $E_4$, $E_5$, and $C_w$. This lattice has bilinear form 
\[
\left(\begin{matrix}
-6 &0 &0 &3\\ 
0 &-2 &1 &0\\ 
0 &1 &-2 &1\\ 
3 &0 &1 &-2\\           
\end{matrix}\right) 
\]  which has discriminant form $\omega_{3,1}^1$. 
By Proposition \ref{p:overl}, there are no non--trivial even overlattices of this lattice, hence $L=L_\cB=S_X(\sigma_9)$. Thus we have the invariants $(r,q_{S_X(\sigma_9)})=(4,w_{3,1}^{1})$. In fact, $S_X(\sigma_9)\cong U\oplus A_2$.
\end{ex}

\begin{rem}
This method also yields some other interesting facts regarding the Picard lattice of these surfaces. 
In \cite{belcastro}, Belcastro computes the Picard lattice for a generic hypersurface with these weights and degree as $U\oplus D_4$. However, if we look at the non--symplectic automorphism $\sigma_9^3$, we 
can compute the invariants $g=1,n=4,k=1$, and therefore $r=10, a=4$ for the invariant lattice, giving us the invariant lattice $S_X(\sigma_9^3)=U\oplus A_2 \oplus E_6$. This shows us in particular, that the Picard lattice of this 
surface is bigger than the Picard lattice for a generic quasihomogeneous polynomial with these weights. 
\end{rem}

\begin{ex}\label{ex-method2}Method II:

In order to illustrate Method II, we repeat the computations for the BHK mirror of the previous example:
\[
W^T=x^2+y^3w+z^9+w^{12}
\] 
with weight system $(18,11,4,3;36)$. This is number 43a in Table~\ref{tab-9}. Here again $|\SL_W/J_W|=1$. As before, we also have an involution, but we consider only $\sigma_9^T$.  

Looking at the action of $\Cstar$ on $\CC^4$ and resolving the singularities we have 
an $A_{10}$ given by resolving the point $(0,1,0,0)$,
$2A_2$ coming from the two points with $y=z=0$ fixed by $\mu_3$ and 
an $A_1$ coming from the point with $y=w=0$ fixed by $\mu_2$.
This time $\cE$ has 15 curves and $\cK=\set{C_x,C_y,C_z}$ as in Figure~\ref{fig:res9}. Again we do not depict the three points of intersection between $C_x$ and $C_y$.  

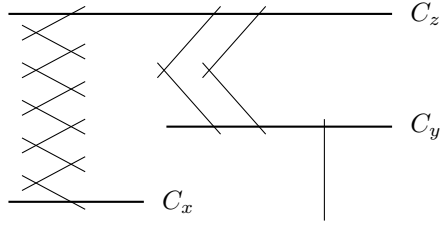
\begin{figure}[ht]
\centering\begin{tikzpicture}[xscale=.6,yscale=.5]
\draw [thick](2,1)--(7,1);
\node [right] at (7.2,1){$C_y$};
\draw [thick](-1.5,4)--(7,4);
\node [right] at (7.2,4){$C_z$};
\draw [thick] (-1.5,-1)--(1.5,-1);
\node [right] at (1.7,-1){$C_x$};
\draw (3.2,4.2)--(1.8,2.3);
\draw (1.8,2.7)--(3.2,0.8);
\draw (4.2,4.2)--(2.8,2.3);
\draw (2.8,2.7)--(4.2,0.8);
\draw (5.5,1.2)--(5.5,-1.5);
\draw (0.2,-1.2)--(-1.2,-0.3);
\draw (-1.2,-0.7)--(0.2,0.2);
\draw (0.2,-0.2)--(-1.2,0.7);
\draw (-1.2,0.3)--(0.2,1.2);
\draw (0.2,0.8)--(-1.2,1.7);
\draw (-1.2,1.3)--(0.2,2.2);
\draw (0.2,1.8)--(-1.2,2.7);
\draw (-1.2,2.3)--(0.2,3.2);
\draw (0.2,2.8)--(-1.2,3.7);
\draw (-1.2,3.3)--(0.2,4.2);

\end{tikzpicture}
\caption{Resolution of singularities on $X_{W^T}$}
\label{fig:res9}
\end{figure}

Two relevant representatives of $\sigma_9^T$ in $G_{W^T}/J_{W^T}$ are $(0,0,\tfrac 19,0)$ and $(0,\tfrac 49,0, \tfrac 23)$. From these we see that the curve $C_z$ is fixed. It has genus 0. 
Furthermore the exceptional curve from the $A_1$ singularity at $y=w=0$ is fixed pointwisely, as well as one of the curves in the exceptional $A_{10}$. Since $C_z$ is fixed, each of the exceptional curves is left invariant under $\sigma_9^T$. From Lemma \ref{l:rank}, the rank of the invariant lattice $S_{X^T}(\sigma_9^T)$ is $r=16$.

In this case, we compute the invariant lattice for $(\sigma_9^T)^3$, which is a non--symplectic automorphism of order 3. 
The curves $C_z$ and $C_y$ are fixed; both have genus zero. Three of the curves on the $A_{10}$ chain are also fixed, as in Lemma~\ref{l:rationaltree}. Furthermore, the remainin intersection points of the chains of exceptional curves are fixed, and an additional point on the $A_1$. Thus the invariants are $(g,n,k)=(1,4,7)$. 
Using the results cited in Section \ref{s:methodsprime}, the invariants for the 3--elementary lattice $(\sigma_9^T)^3$ are $(16,1)$. Since $S_{X^T}(\sigma_9^T)$ is a primitive sublattice of this 3--elementary lattice, and both have the same rank, they are equal. Therefore we have invariants $(r,q)=(16, w_{3,1}^{-1})$ and the lattice is $S_X(\sigma_9)=U\oplus E_6\oplus E_8$. 

\end{ex}

Comparing the ranks, and noticing that $\omega_{3,1}^1=-\omega_{3,1}^{-1}$, we see the BHK mirror matches the LPK3 mirror symmetry.

\begin{ex} Method III:\label{ex-meth3}

Let $W:=x^2+y^4+yz^4+w^{16}$ with $m=16$ in weight system $(8,4,3,1;16)$. This is number 37b in Table \ref{tab-16}. The order of $\SL_W/J_W$ is 2. This appears to be the same K3 surface investigated in \cite[Example~3.2]{order_sixteen}.  
Computing singularities we obtain an $A_2$ at the point $(0:0:1:0)$ and 
two $A_3$'s at the two points with $z=w=0$.
Resolving these, we obtain the configuration of curves showed in Figure \ref{fig:resIII}. The curves $C_x$ and $C_z$ intersect in four points, which are not depicted. 
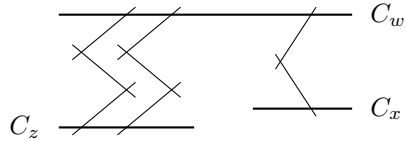
\begin{figure}[ht]
\centering\begin{tikzpicture}[xscale=.6,yscale=.5]

\draw [thick](-1.5,4)--(5,4);
\node [right] at (5.2,4){$C_w$};
\draw [thick] (-1.5,1)--(1.5,1);
\node [left] at (-1.7,1){$C_z$};
\draw [thick] (2.8,1.5)--(5,1.5);
\node [right] at (5.2,1.5){$C_x$};
\draw (4.2,4.2)--(3.3,2.55);
\draw (3.3,2.95)--(4.2,1.3);

\draw (1.2,4.2)--(-0.2,2.8);
\draw (-0.2,3.2)--(1.2,1.8);
\draw (1.2,2.2)--(-0.2,0.8);

\draw (0.2,4.2)--(-1.2,2.8);
\draw (-1.2,3.2)--(0.2,1.8);
\draw (0.2,2.2)--(-1.2,0.8);
\end{tikzpicture}
\caption{Resolution of singularities for $X_W$}
\label{fig:resIII}
\end{figure}
The genus of the curve $C_w$ is 0, the genus of $C_z$ is 1, the genus of $C_x$ is 6. However, $C_y$ consists of two components, each a copy of $\PP^1$.

The automorphism $\sigma_{16}=(0,0,0,\tfrac{1}{16})$ fixes $C_w$, and therefore leaves all of the exceptional curves invariant. Thus we have $|\cE/\sigma_{16}|=8$, and $r=9$. Furthermore, $\cK$ consists of the curves $C_w$, $C_z$ and the two curves that make up $C_y$. 
Using an explicit form of the intersection matrix, one can check that the lattice $L_\cB$ is actually generated by the exceptional curves and $C_w$. 
one sees that it is a lattice of type $T_{3,4,4}$. The discriminant group of $T_{3,4,4}$ is $\mathbb Z/8\mathbb Z$ and the corresponding form $q$ is $w^{5}_{2,3}$. This form has one overlattice. However, Belcastro has computed the Picard Lattice for a general member of the family of K3 surfaces with this weight system as $T_{3,4,4}$. Thus $L_\cB\cong T_{3,4,4}$ embeds primitively into $S_X(\sigma_{16})$ and so they are equal, e.g. $S_X(\sigma_{16})=L_\cB$ with invariants $(r,q)=(9,w^5_{2,3})$. 

\end{ex}

\begin{rem}
There is another case with with the same invariant lattice in the same weight system, namely number 37a. The reasoning is similar to what we have just outlined. 
\end{rem}

Finally, we will describe both of the cases requiring what we have called Method IV. These two cases are similar in that we use the Picard lattice to help determine $S_X(\sigma_{m})$. We will need the following proposition. 

\begin{prop}[{\cite[Prop. 1.15.1]{nikulin}}]\label{p:primembedding}
The primitive embeddings of a lattice $L$ into an even lattice with invariants $(m_+,m_-,q)$ are determined by the sets $(H_L, H_q,\gamma;K,\gamma_K)$, where $H_L\subset A_L$ and $H_q\subset A_q$ are subgroups, $\gamma:q_S|_{H_S}\to q|_{H_q}$ is an isomorphism of subgroups preserving the quadratic forms to these subgroups, $K$ is an even lattice with invariants $(m_+-t_+, m_--t_-, -\delta)$, where $\delta\cong q_S\oplus (-q)|_{\Gamma_\gamma^\perp/\Gamma_\gamma}$, $\Gamma_\gamma$ being the pushout of $\gamma$ in $A_S\oplus A_q$, and, finally, $\gamma_K:q_K\to (-\delta)$ is an isomorphism of quadratic forms. 
\end{prop}

From this proposition, we can determine all primitive embeddings of one even lattice into another. We will use this in the next example. 

\begin{ex}Method IV:\label{ex-meth41}

We now consider the BHK dual to the previous example. This is the first entry for 37b in Table~\ref{tab-16}. As mentioned in the introduction, this provides an example to the case in \cite{order_sixteen}, where no example could be found. In this case, we have 
\[
W^T=W=x^2+y^4+yz^4+w^{16},
\]
and $G^T=\SL_W$, and we know $|\SL_W/J_W|=2$. In fact the group is generated by $(\tfrac{1}{2},0, \tfrac{1}{2},0)$, so we see that the points with $x=z=0$ are fixed. Another representative in the same coset is $(\tfrac{1}{2},0, 0, \tfrac{1}{2})$. Thus we see that the intersection points on the $A_2$ chain from the previous example are fixed, as well as other point with $x=w=0$. The two $A_3$ chains are permuted by the action. Thus on $X_{W^T,G^T}$ we get the configuration of curves of Figure \ref{fig:res16}.

\begin{figure}[ht]
\centering\begin{tikzpicture}[xscale=.6,yscale=.5]

\draw [thick](-1.5,4)--(5,4);
\node [right] at (5.2,4){$C_w$};
\draw [thick] (-1.5,1)--(1.5,1);
\node [left] at (-1.7,1){$C_z$};
\draw [thick] (0,-1)--(5,-1);
\node [right] at (5.2,-1){$C_x$};
\draw (3.2,4.2)--(2.3,2.8);
\draw (2.3,3.2)--(3.2,1.8);
\draw (3.2,2.2)--(2.3,0.8);
\draw (2.3,1.2)--(3.2,-0.2);
\draw (3.2,0.2)--(2.3,-1.2);

\draw (0.2,4.2)--(-1.2,2.8);
\draw (-1.2,3.2)--(0.2,1.8);
\draw (0.2,2.2)--(-1.2,0.8);

\draw (4,4.2)--(4,-1.2);

\draw (1,1.2)--(1,-1.2);
\draw (.8,1.2)--(.8,-1.2);
\draw (.6,1.2)--(.6,-1.2);
\draw (.4,1.2)--(.4,-1.2);

\end{tikzpicture}
\caption{Resolution of singularities on $X_{W^T,G^T}$}
\label{fig:res16}
\end{figure}
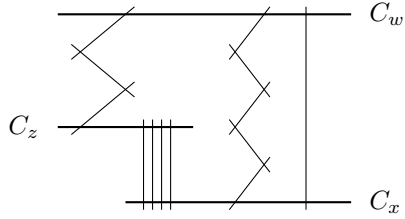

Using the Riemann--Hurwitz Theorem, we can compute the genus of the coordinate curves. The curve $C_x$ is covered by a curve of genus 6 with 6 fixed points, so it has genus 2. Similarly, we see that the genus of $C_w$ is 0 and the genus of $C_z$ is 0. 
The two components of the curve $\set{y=0}$ from the previous example are permuted, to give us $C_y$ of genus 0. 

The non--symplectic automorphism $\sigma_{16}=(0,0,0,\tfrac{1}{16})$, fixes $C_w$, and therefore the chains of exceptional curves intersecting $C_w$ are invariant. It is not difficult to see also that the four exceptional curves intersecting $C_x$ and $C_z$ are permuted. Thus $r=11$. 
One can check that $C_z$, $C_y$ and $C_x$ are superfluous, giving us a lattice $L_\cB$ 
%
with discriminant form $w^{-5}_{2,3}$. 

There is one isotropic subgroup $H$ and hence one overlattice of $L_\cB$. By Proposition~\ref{p:overl} this overlattice has discriminant form $\omega^{-1}_{2,1}$.
Since $S_X(\sigma_{16})$ is an overlattice of $L_\cB$, the two possibilities for $S_X(\sigma_{16})$ are $U\oplus E_8\oplus A_1$ or $T_{2,5,6}$. Using Proposition~\ref{p:primembedding} we will show that $U\oplus E_8\oplus A_1$ does not embed primitively into $S_{X_{W^T,G^T}}$, so that $S_X(\sigma_{16})=T_{2,5,6}$. 

In \cite{order_sixteen} Al Tabbaa-Sarti-Taki have computed the Picard Lattice for K3 surfaces with non--symplectic automorphisms of order 16, and found that in our case, the Picard lattice is $U(2)\oplus D_4\oplus E_8$. This lattice is 2--elementary with $u\oplus v$ as discriminant quadratic form. In particular, this quadratic form takes values 0 or 1 (i.e. $\delta=0$). 

On the other hand, $\omega_{2,1}^{-1}$ has value $\tfrac{3}{2}$ on the generator for $\ZZ/2\ZZ$. By Proposition~\ref{p:primembedding}, a primitive embedding of $U\oplus E_8\oplus A_1$ into the Picard lattice $U(2)\oplus D_4\oplus E_8$ must therefore correspond to the trivial subgroup. The existence of such a primitive embedding depends on the existence of an even lattice with invariants $(0,3,u\oplus v\oplus \omega_{2,1}^1)$. The length of this discriminant quadratic form is 5, whereas the rank of the desired lattice is 3, and so no such lattice exists (see \cite[Theorem 1.10.1]{nikulin}). 

We conclude that the invariant lattice is $S_X(\sigma_{16})\cong T_{2,5,6}$, which has invariants $(11, \omega_{2,3}^{-5})$. 
\end{ex}

\begin{rem}
The other case with $m=16$, $r=11$ is number 58 in Table~\ref{tab-16}. The method for computing the invariant lattice in this case is very similar to what we have just computed. Alternatively, that case can also be computed with Method III. 
\end{rem}

\begin{ex}Method IV:\label{ex-meth42}

The other case that requires Method IV is $m=9$, $r=12$. This occurs for two of the K3 surfaces, namely 18a and 18b, both instances using the group $\SL_W$. Both of these cases are similar, so we describe only the first. 

Using methods similar to those described in the previous examples, we get the configuration of curves depicted in Figure~\ref{fig:resoIV}.   
\begin{figure}[ht]
\centering\begin{tikzpicture}[xscale=.6,yscale=.5]
\draw [thick](-1,-2)--(6,-2);
\node [left] at (-1.2,-2){$C_z$};
\draw [thick](-1,2)--(6,2);
\node [left] at (-1.2,2){$C_w$};
\draw [thick] (1,0)--(6,0);
\node [right] at (6.2,0){$C_x$};
\draw (-1.2,0.2)--(0.2,-2.2);
\draw (-1.2,-.2)--(0.2,2.2);
\draw (3.1,0.8)--(4,2.2);
\draw (4,-0.3)--(3.1,1.2);

\draw (4.3,1.5)--(5.2,2.2);
\draw (5.2,1.1)--(4.3,1.7);
\draw (4.3,0.7)--(5.2,1.3);
\draw (5.2,0.3)--(4.3,0.9);
\draw (4.3,-0.2)--(5.2,0.5);

\draw (1.3,0.2)--(.6,-1.2);
\draw (.6,-0.8)--(1.3,-2.2);

\draw (1.8,0.2)--(1.1,-1.2);
\draw (1.1,-0.8)--(1.8,-2.2);

\draw (2.3,0.2)--(1.6,-1.2);
\draw (1.6,-0.8)--(2.3,-2.2);

\end{tikzpicture}
\caption{Resolution of curves on $X_{W,G}$ for example 18a}
\label{fig:resoIV}
\end{figure}
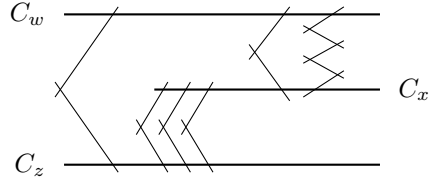
For the discussion that follows, we denote by $E_1$ the exceptional curve in the $A_5$ chain, which intersects $C_x$. 

The automorphism $\sigma_9$ permutes the three $A_2$ chains (yielding one orbit for each curve in the chain for a total of 2 orbits), but leaves the other nine exceptional curves, as well as the coordinate curves, invariant. This gives us $r=12$, and we can compute the lattice $L_\cB\cong M_9\oplus A_2\oplus E_8$ with discriminant quadratic form $\omega_{3,1}^1\oplus\omega_{3,2}^1$. There is one isotropic subgroup of this lattice, corresponding to the overlattice $U\oplus A_2\oplus E_8$. So we must find some way to show that $L_\cB$ is primitively embedded in the Picard lattice, for then it is the invariant lattice.  

We can determine the Picard lattice $S_{X_{W,G}}$. We first notice that $\sigma_9^3$ has order 3. Furthermore, its fixed locus has invariants $(g,n,k)=(0,3,7)$. Therefore its invariant lattice $S_X(\sigma_9^3)$ is a 3-elementary lattice with invariants $(16,3)$, i.e. it is the lattice $U\oplus E_8\oplus 3A_2$, with discriminant quadratic form $3\omega_{3,1}^{1}$. Since the transcendental lattice $S_{X_{W,G}}^\perp$ has order divisible by $\phi(9)=6$, $S_X(\sigma_9^3)$ is the Picard lattice. 

In fact, we 
will determine a basis for $S_{X_{W,G}}$. Consider the set $\cE_1$ consisting of all of the exceptional curves (15 of them), and $\cK$ consisting of all irreducible components of (the strict transforms of) the coordinate curves. Let $\cB_1=\cE_1\cup\cK$. This set generates $S_{X_{W,G}}$. One can check by direct computation (\cite{magma}) that $C_x$ and $E_1$ are redundant. 

Now we consider the set $\cB$, generating the lattice $L_\cB$. Again, we compute that $C_x$ and $E_1$ are redundant, so we get $L_\cB$ generated by the two orbits from the $A_2$ chains, the remaining exceptional curves, and $C_w$ and $C_z$. Two of the generators for $L_\cB$ are just sums of generators of $S_{X_{W,G}}$. Thus a change of basis shows that $S_X/L_\cB$ is a free group of rank 4, and so $L_\cB$ is primitively embedded. 
\end{ex}

The other example is similar. Instead of three $A_2$ chains, there are three $A_1$'s. One can check that the set $\set{y=0}$ is composed of three curves, each of genus zero. Each of these curves intersects one of the $A_1$ curves. These are permuted by the action of $\sigma_9$. Up to a relabelling, we obtain the same configuration of curves, and the same Picard lattice.  






\label{ex-noncyclic}

The only cases where $\SL_W/J_W$ is not cyclic are $x^3+y^3+z^6+w^6$ (number 3d) and $x^2+y^4+z^6+w^{12}$ (number 8d) in Table~\ref{tab-6}.
We analyze them separately.

\begin{ex}
The first polynomial we consider is $W=x^3+y^3+z^6+w^6$ in $\PP(2,2,1,1)$. This is number 3d in Table~\ref{tab-6}.

In this case the order of $\SL_W/J_W$ is 9 and it results that $\SL_W/J_W=\ZZ/3\ZZ\oplus\ZZ/3\ZZ$ since there are no elements of order 9.  
The group $J_W$ is generated by $j_W=(\frac12,\frac13,\frac16,\frac16)$ and two generators for $\SL_W/J_W$ are $g_1=(\frac13,\frac23,0,0)$ and $g_2=(\frac13,\frac13,\frac13,0)$.
Also we name $g_3=(\frac13,0,\frac23,0)$ and $g_4=(0,\frac13,0,\frac23)$.
There are four subgroups of $\SL_W/J_W$ of order 3, 
namely $G_i=<g_i,\J{W}>, i=1,2,3,4$.
We can also observe that $G_1^T=G_2$, $G_3^T=G_3$ and $G_4^T=G_4$.

Now we consider the non-symplectic automorphism $\sigma_{6}=(0,0,0,\tfrac{1}{6})$. One may notice, there is another automorphism of order 6, namely $(0,0,\tfrac{1}{6},0)$, but due to symmetry (i.e. exchanging $z$ and $w$), we must only consider one of them. 

Using the same methods as before for each group, we can compute the invariant lattice for the corresponding K3 surface. In each case, the lattice $L_\cB$ has no overlattices, so we can use Method I. When $G=G_3$, the invariant lattice has $(r,q)=(10,4w_{3,1}^1)$, which is self-dual. The same is true for $G=G_4$.

When $G=G_1$ we get an invariant lattice with rank 16 and the discriminant form is $v\oplus w_{3,1}^1$. This is the dual of the invariant lattice we get with the choice $G=G_2$ and so it proves the theorem for this case.
\end{ex}

\begin{ex}
Finally, we examine the polynomial $W=x^2+y^4+z^6+w^{12}$ in $\PP(6,3,2,1)$. This is number 8d in Table~\ref{tab-6}. There are non--symplectic automorphisms of order 2, 4 and 12, but we again focus on the non-symplectic automorphism of order 6: $\sigma_6=(0,0,\tfrac{1}{6},0)$.  


The order of $\SL_W/J_W$ is 4 and since there are no elements of order 4, we conclude that $\SL_W/J_W=\ZZ/2\ZZ\oplus\ZZ/2\ZZ$.
The elements
\[g_1=\left(\frac12,0,\frac12,0\right),\ g_2=\left(0,\frac12,\frac12,0\right),\ g_3=\left(\frac12,\frac12,0,0\right)\]
each have order 2 and represent different cosets in $\SL_W/J_W$; let $G_i:=<g_i,\J{W}>,i=1,2,3$.
Observe that $G_1^T=G_1$ while $G_2^T=G_3$.

When $G=G_1$, with Method I we compute the invariant lattice and obtain $(r,q)=(10,v\oplus 4w_{2,1}^{-1})$.

As for $G_2$, we use again Method I and obtain $(r,q)=(14,2w_{2,1}^{-1}\oplus w_{3,1}^1)$, while for $G_3$ we get $(r,q)=(6,2w_{2,1}^{1}\oplus w_{3,1}^{-1})$.
Observing they are mirror of each other, we can conclude that the Theorem is proved in this case.
\end{ex}

\section{Tables}
\label{sec:tables}
\footnotesize{
In each table, we have arranged the surfaces by weight system. Each weight system is listed by the number assigned to it by Yonemura in \cite{yonemura}. In each weight system, we have listed all possible invertible polynomials of the form \eqref{eq-W} with non--symplectic automorphism of order $m$, and for each polynomial, we list the orders of the possible groups $G/J_W$ satisfying $J_W\subseteq G\subseteq \SL_W$. The invariants $(r,q_{S_X(\sigma_m)})$ are then given, as well as the number of the BHK mirror dual. Finally, we have also indicated which method was used to determine $q_{S_X(\sigma_m)}$.  

When consulting the tables, it will be helpful to know that $\omega_{5,1}^\epsilon=-\omega_{5,1}^\epsilon$, and that $4\omega_{3,1}^{-1}=4\omega_{3,1}^1$, $4\omega_{2,1}^{-1}=4\omega_{2,1}^1$. The first fact follows simply by definition. The latter two follow from \cite[Theorem 1.8.2]{nikulin} 

\begin{longtable}{p{12pt}|c|c|c|c|c|c|c}
\multicolumn{7}{c}{}\\
No. &Weights&Polynomial&SL/J&G/J&$(r,q)$&BHK dual&Method\\
 \hline
  \midrule
 \endfirsthead
 
 No. &Weights&Polynomial&SL/J&G/J&$(r,q)$&BHK dual&Method\\\hline    \midrule
 \endhead
14& (21,14,6,1;42)&	$x^2+y^3+z^7+w^{42}$&	1&1&$(10,<0>)$&14&I\\

\caption{Table for $m=42$} \label{tab-42} \end{longtable}

\begin{longtable}{p{12pt}|c|c|c|c|c|c|c}
\multicolumn{7}{c}{}\\
No. &Weights&Polynomial&SL/J&G/J&$(r,q)$&BHK dual&Method\\
 \hline
  \midrule
 \endfirsthead
 
 No. &Weights&Polynomial&SL/J&G/J&$(r,q)$&BHK dual&Method\\\hline    \midrule
 \endhead
 38&	(15,8,6,1;30)		&$x^2+y^3z+z^5+w^{30}$&		1&1&$(11,w_{2,1}^{-1})$	&50&I	\\
 50&	(15,10,4,1;30)		&$x^2+y^3+yz^5+w^{30}$&		1&1&$(9,w_{2,1}^{1})$	&38&I	\\	 
 \caption{Table for $m=30$} \label{tab-30} \end{longtable}

\begin{longtable}{p{12pt}|c|c|c|c|c|c|c}
\multicolumn{7}{c}{}\\
No. &Weights&Polynomial&SL/J&G/J&$(r,q)$&BHK dual&Method\\
 \hline
  \midrule
 \endfirsthead
 No. &Weights&Polynomial&SL/J&G/J&$(r,q)$&BHK dual&Method\\\hline    \midrule
 \endhead
13a	&(12,8,3,1;24)	&$x^2+y^3+xz^4+w^{24}$	&1	&1	&$(8,w_{3,1}^{-1})$	&20&I\\
13b	&(12,8,3,1;24)	&$x^2+y^3+z^8+w^{24}$	&2	&2	&$(12, w_{3,1}^{1})$	&13b&I\\
	&			&					&	&1	&$(8,w_{3,1}^{-1})$	&13b&I\\
20	&(9,8,6,1;24)	&$x^2z+y^3+z^4+w^{24}$	&1	&1	&$(12, w_{3,1}^{1})$	&13a&I\\	 
\caption{Table for $m=24$} \label{tab-24} 
\end{longtable}

\begin{longtable}{p{12pt}|c|c|c|c|c|c|c}
\multicolumn{7}{c}{}\\
No. &Weights&Polynomial&SL/J&G/J&$(r,q)$&BHK dual&Method\\
 \hline
  \midrule
 \endfirsthead
 No. &Weights&Polynomial&SL/J&G/J&$(r,q)$&BHK dual&Method\\\hline    \midrule
 \endhead
 78	&(11,6,4,1;22)	&$x^2+y^3z+yz^4+w^{22}$	&1	&1	&$(10,w_{2,1}^{-1}\oplus w_{2,1}^{1})$	&78&II\\
\caption{Table for $m=22$} \label{tab-22}  \end{longtable}

\begin{longtable}{p{12pt}|c|c|c|c|c|c|c}
\multicolumn{7}{c}{}\\
No. &Weights&Polynomial&SL/J&G/J&$(r,q)$&BHK dual&Method\\
 \hline
  \midrule
 \endfirsthead
 No. &Weights&Polynomial&SL/J&G/J&$(r,q)$&BHK dual&Method\\\hline    \midrule
 \endhead
 9a	&(10,5,4,1;20)	&$x^2+xy^2+z^5+w^{20}$	&1	&1	&$(10,w_{5,1}^{-1})$		&9a&I\\
9b	&(10,5,4,1;20)	&$x^2+y^4+z^5+w^{20}$	&2	&2	&$(10,w_{5,1}^{-1})$		&9b&I\\
	&			&					&	&1	&$(10,w_{5,1}^{-1})$		&9b&I\\
 \caption{Table for $m=20$} \label{tab-20} \end{longtable}

\begin{longtable}{p{12pt}|c|c|c|c|c|c|c}
\multicolumn{7}{c}{}\\
No. &Weights&Polynomial&SL/J&G/J&$(r,q)$&BHK dual&Method\\
 \hline
  \midrule
 \endfirsthead
 No. &Weights&Polynomial&SL/J&G/J&$(r,q)$&BHK dual&Method\\\hline    \midrule
 \endhead
12a	&(9,6,2,1;18)	&$x^2+y^3+yz^6+w^{18}$		&2	&2	&$(11,w_{2,1}^1\oplus w_{3,1}^1)$		&39a&I\\
	&			&						&	&1	&$(6,v)$							&39a&I\\
12b	&(9,6,2,1;18)	&$x^2+y^3+z^9+w^{18}$		&3	&3	&$(14,v)$							&12b&I\\
	&			&						&	&1	&$(6,v)$							&12b&I\\
39a	&(9,5,3,1;18)	&$x^2+y^3z+z^6+w^{18}$		&2	&2	&$(14,v)$							&12a&I\\
	&			&						&	&1	&$(9,w_{2,1}^{-1}\oplus w_{3,1}^{-1})$	&12a&I\\
39b	&(9,5,3,1;18)	&$x^2+y^3z+xz^3+w^{18}$	&1	&1	&$(9,w_{2,1}^{-1}\oplus w_{3,1}^{-1})$	&60&I\\
60	&(7,6,4,1;18)	&$x^2z+y^3+yz^3+w^{18}$	&1	&1	&$(11,w_{2,1}^1\oplus w_{3,1}^1)$		&39b&I\\

\caption{Table for $m=18$} \label{tab-18}  \end{longtable}

\begin{longtable}{p{12pt}|c|c|c|c|c|c|c}
\multicolumn{7}{c}{}\\
No. &Weights&Polynomial&SL/J&G/J&$(r,q)$&BHK dual&Method\\
 \hline
  \midrule
 \endfirsthead
 No. &Weights&Polynomial&SL/J&G/J&$(r,q)$&BHK dual&Method\\\hline    \midrule
 \endhead
37a	&(8,4,3,1;16)	&$x^2+xy^2+yz^4+w^{16}$	&1	&1	&$(9,w_{2,3}^{5})$					&58&III\\
37b	&(8,4,3,1;16)	&$x^2+y^4+yz^4+w^{16}$		&2	&2	&$(11,w_{2,3}^{-5})$					&37b&IV\\
	&			&						&	&1	&$(9,w_{2,3}^{5})$					&37b&III\\
58	&(6,5,4,1;16)	&$x^2z+xy^2+z^4+w^{16}$	&1	&1	&$(11,w_{2,3}^{-5})$					&37a&IV\\
\caption{Table for $m=16$} \label{tab-16}  \end{longtable}



\begin{longtable}{p{12pt}|c|c|c|c|c|c|c}
\multicolumn{7}{c}{}\\
No. &Weights&Polynomial&SL/J&G/J&$(r,q)$&BHK dual&Method\\
 \hline
  \midrule
 \endfirsthead
 No. &Weights&Polynomial&SL/J&G/J&$(r,q)$&BHK dual&Method\\\hline    \midrule
 \endhead
11a	&(15,10,3,2;30)		&$x^2+y^3+xz^5+w^{15}$ 	&1	&1	&$(10,w_{3,1}^{-1}\oplus w_{3,1}^{1})$	&22a&II\\
11b	&(15,10,3,2;30)		&$x^2+y^3+z^{10}+w^{15}$		&1	&1	&$(10,w_{3,1}^{-1}\oplus w_{3,1}^{1})$	&11b&II\\
22a	&(6,5,3,1;15)		&$x^2z+y^3+z^5+w^{15}$		&1	&1	&$(10,w_{3,1}^{-1}\oplus w_{3,1}^{1})$	&11a&II\\
22b	&(6,5,3,1;15)		&$x^2z+y^3+xz^3+w^{15}$	&1	&1	&$(10,w_{3,1}^{-1}\oplus w_{3,1}^{1})$	&22b&II\\
 \caption{Table for $m=15$} \label{tab-15} 
 \end{longtable}
 


\begin{longtable}{p{12pt}|c|c|c|c|c|c|c}
 \multicolumn{7}{c}{}\\
No. &Weights&Polynomial&SL/J&G/J&$(r,q)$&BHK dual&Method\\
 \hline
  \midrule
 \endfirsthead
 No. &Weights&Polynomial&SL/J&G/J&$(r,q)$&BHK dual&Method\\\hline    \midrule
 \endhead
40a	&(7,4,2,1;14)	&$x^2+y^3z+z^{7}+w^{14}$	&1	&1	&$(7,v\oplus w_{2,1}^{-1})$	&47	&III\\
40b	&(7,4,2,1;14)	&$x^2+y^3z+yz^{5}+w^{14}$	&2	&2	&$(13,v\oplus w_{2,1}^{1})$	&40b&II	\\
	&			&						&	&1	&$(7,v\oplus w_{2,1}^{-1})$	&40b	&III\\
47	&(21,14,4,3;42)	&$x^2+y^3+yz^{7}+w^{14}$	 &1	&1	&$(13,v\oplus w_{2,1}^{1})$	&40a&II	\\
 \caption{Table for $m=14$} \label{tab-14} \end{longtable}



\begin{longtable}{p{12pt}|c|c|c|c|c|c|c}
\multicolumn{7}{c}{}\\
No. &Weights&Polynomial&SL/J&G/J&$(r,q)$&BHK dual&Method\\
 \hline
  \midrule
 \endfirsthead
 No. &Weights&Polynomial&SL/J&G/J&$(r,q)$&BHK dual&Method\\\hline    \midrule
 \endhead
6a	&(5,2,2,1;10)	&$x^2+y^4z+z^5+w^{10}$		&2	&2	&$(8,w_{5,1}^{-1}\oplus 2w_{2,1}^1)$	&36a&I	\\
	&			&						&	&1	&$(6,u\oplus v)$					&36a&II	\\
6b	&(5,2,2,1;10)	&$x^2+y^5+z^5+w^{10}$		&5	&5	&$(14,u\oplus v)$					&6b	&II\\
	&			&						&	&1	&$(6,u\oplus v)$					&6b	&II\\
6c	&(5,2,2,1;10)	&$x^2+y^4z+yz^4+w^{10}$	&3	&3	&$(14,u\oplus v)$					&6c	&II\\
	&			&						&	&1	&$(6,u\oplus v)$					&6c	&II\\
11a	&(15,10,3,2;30)	&$x^2+y^3+z^{10}+yw^{10}$	&2	&2	&$(17,w_{2,1}^{1})$				&42a&I	\\
	&			&						&	&1	&$(10,v\oplus v)$					&42a&II	\\	
11b	&(15,10,3,2;30)&$x^2+y^3+z^{10}+w^{15}$		&1	&1	&$(10,v\oplus v)$					&11b&II	\\

36a	&(10,5,3,2;20)	&$x^2+y^4+yz^5+w^{10}$		&2	&2	&$(14,u\oplus v)$				&6a	&II\\
	&			&						&	&1	&$(12,w_{5,1}^{-1}\oplus 2w_{2,1}^{-1})$					&6a&I	\\
36b	&(10,5,3,2;20)	&$x^2+xy^2+yz^5+w^{10}$	&1	&1	&$(12,w_{5,1}^{-1}\oplus 2w_{2,1}^{-1})$					&63&I	\\
42a	&(5,3,1,1;10)	&$x^2+y^3w+z^{10}+w^{10}$	&2	&2	&$(10,v\oplus v)$					&11a&II	\\
	&			&						&	&1	&$(3,w_{2,1}^{-1})$				&11a&I	\\
42b	&(5,3,1,1;10)	&$x^2+y^3z+xz^5+w^{10}$	&1	&1	&$(3,w_{2,1}^{-1})$				&68&I	\\
42c	&(5,3,1,1;10)	&$x^2+y^3z+yz^7+w^{10}$	&4	&4	&$(17,w_{2,1}^{1})$				&42c&I	\\
	&			&						&	&2	&$(10,v\oplus v)$					&42c&II\\
	&			&						&	&1	&$(3,w_{2,1}^{-1})$				&42c&I	\\
63	&(4,3,2,1;10)	&$x^2z+y^2x+z^5+w^{10}$	&1	&1	&$(8,w_{5,1}^{-1}\oplus 2w_{2,1}^1)$	&36b&I	\\
68	&(13,10,4,3;30)	&$x^2z+y^3+yz^5+w^{10}$	&1	&1	&$(17,w_{2,1}^{1})$				&42b&I	\\
\caption{Table for $m=10$} \label{tab-10}  \end{longtable}




\begin{longtable}{p{12pt}|c|c|c|c|c|c|c}
\multicolumn{7}{c}{}\\
No. &Weights&Polynomial&SL/J&G/J&$(r,q)$&BHK dual&Method\\
 \hline
  \midrule
 \endfirsthead
 No. &Weights&Polynomial&SL/J&G/J&$(r,q)$&BHK dual&Method\\\hline    \midrule
 \endhead
 12a	&(9,6,2,1;18)	&$x^2+y^3+z^9+xw^9$		&3	&3	&$(16,w_{3,1}^{-1})$				&25a&I	\\
	&			&						&	&1	&$(4,w_{3,1}^1)$				&25a&I	\\
12b	&(9,6,2,1;18)	&$x^2+y^3+z^9+yw^{12}$		&1	&1	&$(4,w_{3,1}^1)$				&43a&I	\\
12c	&(9,6,2,1;18)	&$x^2+y^3+z^9+w^{18}$		&3	&3	&$(16,w_{3,1}^{-1})$				&12c&I	\\
	&			&						&	&1	&$(4,w_{3,1}^1)$				&12c&I	\\
18a	&(3,3,2,1;9)	&$x^3+y^3+xz^3+w^9$		&3	&3	&$(12,w_{3,1}^1\oplus w_{3,2}^1)$		&18a&IV	\\
	&			&						&	&1	&$(8,w_{3,1}^{-1}\oplus w_{3,2}^{-1})$	&18a&III	\\
18b	&(3,3,2,1;9)	&$x^3+xy^2+yz^3+w^9$		&2	&2	&$(12,w_{3,1}^1\oplus w_{3,2}^1)$		&18b&IV	\\
	&			&						&	&1	&$(8,w_{3,1}^{-1}\oplus w_{3,2}^{-1})$	&18b&III	\\
25a	&(4,3,1,1;9)	&$x^2w+y^3+z^9+w^9$		&3	&3	&$(16,w_{3,1}^{-1})$				&12a&I	\\
	&			&						&	&1	&$(4,w_{3,1}^1)$				&12a&I	\\
25b	&(4,3,1,1;9)	&$x^2w+y^3+z^9+yw^6$		&1	&1	&$(4,w_{3,1}^1)$				&43b&I	\\
25c	&(4,3,1,1;9)	&$x^2w+y^3+z^9+xw^5$		&3	&3	&$(16,w_{3,1}^{-1})$				&25a&I	\\
	&			&						&	&1	&$(4,w_{3,1}^1)$				&25a&I	\\
43a	&(18,11,4,3;36)	&$x^2+y^3w+z^9+w^{12}$	&1	&1	&$(16,w_{3,1}^{-1})$				&12b&I	\\
43b	&(18,11,4,3;36)	&$x^2+y^3w+z^9+xw^6$		&1	&1	&$(16,w_{3,1}^{-1})$				&25b&I	\\
\caption{Table for $m=9$} \label{tab-9}  \end{longtable}  
 


\begin{longtable}{p{12pt}|c|c|c|c|c|c|c}
\multicolumn{7}{c}{}\\
No. &Weights&Polynomial&SL/J&G/J&$(r,q)$&BHK dual&Method\\
 \hline
  \midrule
 \endfirsthead
 No. &Weights&Polynomial&SL/J&G/J&$(r,q)$&BHK dual&Method\\\hline    \midrule
 \endhead
2a	&(4,3,3,2;12)	&$x^3+y^3z+z^4+w^6$	&3	&3	&$(16,v\oplus w_{3,1}^1)$				&3a&I	\\
	&			&					&	&1	&$(10,4w_{3,1}^1)$				&3a&II	\\
2b	&(4,3,3,2;12)	&$x^3+y^3z+yz^3+w^6$	&1	&1	&$(10,4w_{3,1}^1)$				&2b	&II\\
2c	&(4,3,3,2;12)	&$x^3+y^4+z^4+w^6$	&2	&2	&$(10,4w_{3,1}^1)$				&2c&II	\\
	&			&					&	&1	&$(10,4w_{3,1}^1)$				&2c&II	\\
	
3a	&(2,2,1,1;6)	&$x^3+y^3+yz^4+w^6$	&3	&3	&$(10,4w_{3,1}^1)$				&2a	&II\\
	&			&					&	&1	&$(4,v\oplus w_{3,1}^{-1})$			&2a	&I\\
3b	&(2,2,1,1;6)	&$x^2y+y^3+z^6+w^6$	&6	&6	&$(19,w_{2,1}^{-1})$				&5a	&I\\
	&			&					&	&3	&$(16,v\oplus w_{3,1}^1)$				&5a&I	\\
	&			&					&	&2	&$(9,3w_{2,1}^{-1}\oplus 2w_{3,1}^{-1})$	&5a&I	\\
	&			&					&	&1	&$(4,v\oplus w_{3,1}^{-1})$			&5a	&I\\
3c	&(2,2,1,1;6)	&$x^3+xy^2+yz^4+w^6$	&1	&1	&$(4,v\oplus w_{3,1}^{-1})$			&57	&I\\
3d	&(2,2,1,1;6)	&$x^3+y^3+z^6+w^6$	&9	&9	&$(16,v\oplus w_{3,1}^1)$				&3d&I	\\
	&			&					&	&3	&$(16,v\oplus w_{3,1}^1)$				&3d	&I\\
	&			&					&	&3	&$(10,4w_{3,1}^1)$				&3d&II	\\
	&			&					&	&3	&$(10,4w_{3,1}^1)$				&3d&II	\\
	&			&					&	&3	&$(4,v\oplus w_{3,1}^{-1})$			&3d	&I\\
	&			&					&	&1	&$(4,v\oplus w_{3,1}^{-1})$			&3d	&I\\
3e	&(2,2,1,1;6)	&$x^2y+xy^2+z^6+w^6$	&3	&3	&$(16,v\oplus w_{3,1}^1)$				&3e	&I\\
	&			&					&	&1	&$(4,v\oplus w_{3,1}^{-1})$			&3e	&I\\
	
5a	&(3,1,1,1;6)	&$x^2+xy^3+z^6+w^6$	&6	&6	&$(16,v\oplus w_{3,1}^1)$				&3b&I	\\
	&			&					&	&3	&$(11,3w_{2,1}^1\oplus 2w_{3,1}^{1})$	&3b&I	\\
	&			&					&	&2	&$(4,v\oplus w_{3,1}^{-1})$			&3b	&I\\
	&			&					&	&1	&$(1,w_{2,1}^1)$				&3b&I	\\
5b	&(3,1,1,1;6)	&$x^2+y^5w+z^6+w^6$	&2	&2	&$(8,6w_{2,1}^{-1})$				&29	&II\\
	&			&					&	&1	&$(1,w_{2,1}^1)$				&29&I	\\
5c	&(3,1,1,1;6)	&$x^2+xy^3+yz^5+w^6$	&1	&1	&$(1,w_{2,1}^1)$				&56	&I\\
5d	&(3,1,1,1;6)	&$x^2+y^6+z^5w+zw^5$	&8	&8	&$(19,w_{2,1}^{-1})$				&5d	&I\\
	&			&					&	&4	&$(12,6w_{2,1}^1)$				&5d&II	\\
	&			&					&	&2	&$(8,6w_{2,1}^{-1})$				&5d	&II\\
	&			&					&	&1	&$(1,w_{2,1}^1)$				&5d&I	\\
5e	&(3,1,1,1;6)	&$x^2+y^6+z^6+w^6$	&12	&12	&$(19,w_{2,1}^{-1})$				&5e	&I\\
	&			&					&	&6	&$(12,6w_{2,1}^1)$				&5e	&II\\
	&			&					&	&4	&$(9,3w_{2,1}^{-1}\oplus 2w_{3,1}^{-1})$	&5e&I	\\
	&			&					&	&3	&$(11,3w_{2,1}^1\oplus 2w_{3,1}^{1})$	&5e	&I\\
	&			&					&	&2	&$(8,6w_{2,1}^{-1})$				&5e&II	\\
	&			&					&	&1	&$(1,w_{2,1}^1)$				&5e	&I\\

8a	&(6,3,2,1;12)	&$x^2+y^4+z^6+xw^6$	&2	&2	&$(14,2w_{2,1}^{-1}\oplus w_{3,1}^1)$	&23	&I\\
	&			&					&	&1	&$(6,2w_{2,1}^1\oplus w_{3,1}^{-1})$	&23	&I\\
8b	&(6,3,2,1;12)	&$x^2+y^4+z^6+yw^{9}$	&2	&2	&$(10,v\oplus 4w_{2,1}^{-1})$			&33a&I	\\
	&			&					&	&1	&$(6,2w_{2,1}^1\oplus w_{3,1}^{-1})$	&33a	&I\\
8c	&(6,3,2,1;12)	&$x^2+xy^2+z^6+yw^{9}$	&1	&1	&$(6,2w_{2,1}^1\oplus w_{3,1}^{-1})$	&70	&I\\
8d	&(6,3,2,1;12)	&$x^2+y^4+z^6+w^{12}$	&4	&4	&$(14,2w_{2,1}^{-1}\oplus w_{3,1}^1)$	&8d&I	\\
	&			&					&	&2	&$(14,2w_{2,1}^{-1}\oplus w_{3,1}^1)$	&8d	&I\\
	&			&					&	&2	&$(6,2w_{2,1}^1\oplus w_{3,1}^{-1})$	&8d&I	\\	
	&			&					&	&2	&$(10,v\oplus 4w_{2,1}^{-1})$			&8d&II	\\
	&			&					&	&1	&$(6,2w_{2,1}^1\oplus w_{3,1}^{-1})$	&8d	&I\\	
8e	&(6,3,2,1;12)	&$x^2+xy^2+z^6+w^{12}$	&2	&2	&$(14,2w_{2,1}^{-1}\oplus w_{3,1}^1)$	&8e&I	\\
	&			&					&	&1	&$(6,2w_{2,1}^1\oplus w_{3,1}^{-1})$	&8e	&I\\
	
23	&(5,3,2,2;12)	&$x^2w+y^4+z^6+w^6$	&2	&2	&$(14,2w_{2,1}^{-1}\oplus w_{3,1}^1)$	&8a	&I\\
	&			&					&	&1	&$(6,2w_{2,1}^1\oplus w_{3,1}^{-1})$	&8a	&I\\
29	&(15,6,5,4;30)	&$x^2+y^5+z^6+yw^6$	&2	&2	&$(19,w_{2,1}^{-1})$				&5b&I	\\
	&			&					&	&1	&$(12,6w_{2,1}^1)$				&5b&II	\\
	
33a	&(9,4,3,2;18)	&$x^2+y^4w+z^6+w^{9}$	&2	&2	&$(14,2w_{2,1}^{-1}\oplus w_{3,1}^1)$	&8b	&I\\
	&			&					&	&1	&$(10,v\oplus 4w_{2,1}^{-1})$			&8b	&II\\
33b	&(9,4,3,2;18)	&$x^2+y^4w+z^6+yw^{7}$&1	&1	&$(10,v\oplus 4w_{2,1}^{-1})$			&33b&II	\\
56	&(11,8,6,5;30)	&$x^2y+y^3z+z^5+w^6$	&1	&1	&$(19,w_{2,1}^{-1})$				&5c	&I\\
57	&(9,6,5,4;24)	&$x^2y+y^4+xz^3+w^6$	&1	&1	&$(16,v\oplus w_{3,1}^1)$				&3c	&I\\
70	&(8,5,3,2;18)	&$x^2w+xy^2+z^6+w^{9}$&1	&1	&$(14,2w_{2,1}^{-1}\oplus w_{3,1}^1)$	&8c	&I\\
\caption{Table for $m=6$} \label{tab-6}    \end{longtable}  

}
\normalsize

\appendix

\section{Computer code for computing lattices}

In order to compute the lattices using the configuration of curves, we used the following Magma code, developed by Antonio Laface and added here with his permission. 

This first function takes an even bilinear form $B$, and outputs generators of the discriminant group and the values of $q_B$ on these generators. 

\begin{alltt}
disc:=function(M)
 S,A,B:=SmithForm(M);
 l:=[[S[i,i],i]: i in [1..NumberOfColumns(S)]| S[i,i] notin {0,1}];
 sA:=Matrix(Rationals(),ColumnSubmatrixRange(B,l[1][2],l[#l][2]));
 for i in [1..#l] do
  MultiplyColumn(\(\sim\)sA,1/l[i][1],i);
 end for;
 Q:=Transpose(sA)*Matrix(Rationals(),M)*sA;
 for i,j in [1..NumberOfColumns(Q)] do
  if i ne j then
   Q[i,j]:=Q[i,j]-Floor(Q[i,j]);
  else
   Q[i,j]:=Q[i,j]-Floor(Q[i,j])+ (Floor(Q[i,j]) mod 2);
  end if;
 end for; 
 return [l[i][1]: i in [1..#l]], Q;
end function;    
\end{alltt}

The next function determines whether a given even bilinear form has overlattices. Input is a matrix $M$ and a number $n$. The output is the subgroup of $A_L$ that takes values equal to $n$ modulo $2\ZZ$. For isotropic subgroups of the discriminant group, use $n=0$.  
\begin{verbatim}
isot:=function(M,n)
 v,U:=disc(M);
 Q:=Rationals();
 A:=AbelianGroup(v);
 return [Eltseq(a) : a in A |
 mod2(Matrix(Q,1,#v,Eltseq(a))*U*Matrix(Q,#v,1,Eltseq(a)))[1,1] eq n];
end function;
\end{verbatim}

The function {\verb mod2 } is as follows: 
\begin{verbatim}
mod2:=function(Q);
 for i,j in [1..Nrows(Q)] do
  if i ne j then Q[i,j]:=Q[i,j]-Floor(Q[i,j]);
  else Q[i,j]:=Q[i,j]-2*Floor(Q[i,j]/2);
  end if;
 end for;
 return Q;
end function;

\end{verbatim}

Finally, the following function compares two discriminant quadratic forms, and lets us know if they are the same finite quadratic form or not. This is not always easy to check due to the relations in Proposition~\ref{t:relations}.
\begin{verbatim}
dicompare:=function(M,Q)
 v,U:=disc(M);
 w,D:=disc(Q);
 if v ne w then return false; end if;
 A:=AbelianGroup(v);
 Aut:=AutomorphismGroup(A);
 f,G:=PermutationRepresentation(Aut);
 h:=Inverse(f);
 ll:=[Matrix(Rationals(),[Eltseq(Image(h(g),A.i)) : i in [1..Ngens(A)]]) : g in G];
 dd:=[mod2(a*U*Transpose(a)) : a in ll];
 return D in dd;
end function;
\end{verbatim}

\bibliographystyle{plain}
\bibliography{references}

\end{document}